\documentclass[11pt,a4paper]{article}
\usepackage[applemac]{inputenc}
\usepackage{lmodern}
\usepackage{pbsi}
\usepackage[T1]{fontenc}
\usepackage[english]{babel}
\usepackage{verbatim}
\usepackage[a4paper,vmargin={3.5cm,3.5cm},hmargin={2.5cm,2.5cm}]{geometry}
\usepackage{amssymb}
\usepackage{savesym}
\usepackage{amsmath}
\savesymbol{iint}
\usepackage{amsthm}
\savesymbol{openbox}
\usepackage{pxfonts}
\restoresymbol{TXF}{iint}
\restoresymbol{TXF}{openbox}
\usepackage{graphicx}
\usepackage{color,soul,wrapfig,float}
\usepackage{pgf}
\usepackage{tikz}
\usepackage{contour}
\usepackage{newfloat}
\usepackage{placeins}
\usepackage{array}
\usepackage{tabularx}
\usepackage{hhline}
\usepackage{marginnote}
\usepackage{subcaption}
  
\usepgflibrary{decorations.pathmorphing}
\usetikzlibrary{patterns,arrows,fit,shapes.geometric,shapes,shapes.callouts,decorations.text,intersections,shadows,calc,through,decorations.pathmorphing,decorations,fadings,angles}
\tikzstyle{none}=[inner sep=0pt]
\tikzstyle{b}=[line width=4pt, draw=white]
\tikzstyle{real}=[circle,fill=black,draw=white, inner sep=2pt, very thick]
\tikzstyle{phantom}=[circle,fill=white,draw=black, inner sep=1pt]
\tikzstyle{geoone}=[very thick, orange]
\tikzstyle{geotwo}=[very thick, blue]
\tikzstyle{geothree}=[ultra thick, green!50!black]
\tikzstyle{fence}=[ultra thick, red]
\tikzstyle{map}=[red, densely dashed]
\tikzstyle{root}=[red, very thick, ->]
\tikzstyle{georight}=[very thick, green!50!black]
\tikzstyle{geoleft}=[very thick, blue]
\tikzstyle{bla}=[]
\tikzset{
        hatch distance/.store in=\hatchdistance,
        hatch distance=10pt,
        hatch thickness/.store in=\hatchthickness,
        hatch thickness=2pt
    }

    \makeatletter
    \pgfdeclarepatternformonly[\hatchdistance,\hatchthickness]{flexible hatch}
    {\pgfqpoint{0pt}{0pt}}
    {\pgfqpoint{\hatchdistance}{\hatchdistance}}
    {\pgfpoint{\hatchdistance}{\hatchdistance}}
    {
        \pgfsetcolor{\tikz@pattern@color}
        \pgfsetlinewidth{\hatchthickness}
        \pgfpathmoveto{\pgfqpoint{0pt}{0pt}}
        \pgfpathlineto{\pgfqpoint{\hatchdistance}{\hatchdistance}}
        \pgfusepath{stroke}
    }

\def\llbracket{[\hspace{-.10em} [ }
\def\rrbracket{ ] \hspace{-.10em}]}

\newcommand{\Q}{\mathcal{Q}} 
\newcommand{\UIHPQ}{\mathcal{H}_\infty} 
\newcommand{\UIHPQS}{\widetilde{\mathcal{H}}_\infty} 
\newcommand{\sQ}{\mathsf{Q}} 
\newcommand{\Ball}{\mathfrak{B}}
\newcommand{\q}{\mathfrak{q}} 
\renewcommand{\P}{\mathbb{P}} 

\newcommand{\Z}[2]{Z^{#1,#2}}
\newcommand{\SAW}[2]{\mathsf{SAW}^{#1,#2}}

\renewcommand{\phi}{\varphi}
\renewcommand{\epsilon}{\varepsilon}

\newcommand{\1}{\mathds{1}}

\newcommand{\dgr}{\mathrm{d}_\mathrm{{gr}}}

\pgfdeclarelayer{edgelayer}
\pgfdeclarelayer{nodelayer}
\pgfdeclarelayer{foreground}
\pgfsetlayers{edgelayer,nodelayer,main,foreground}

\usepackage{geometry}
\usepackage{graphicx}
\usepackage{wrapfig}
\usepackage{placeins}
\usepackage{array}
\usepackage{tabularx}
\usepackage{hhline}
\usepackage{marginnote}
\usepackage{caption}
\usepackage{subcaption}
\usepackage{fixltx2e}
\usepackage{afterpage}

\usepackage{dsfont}
\usepackage{pifont}
\usepackage{pxfonts}

\newtheorem{rem}{Remark}[section]

\theoremstyle{plain}\newtheorem{teo}{Theorem}

\newtheorem{theorem}[teo]{Theorem}

\newtheorem{proposition}[teo]{Proposition}
\newtheorem{cor}[teo]{Corollary}

\newtheorem{conj}{Conjecture}
\newtheorem{open}[]{Open question}
\newtheorem{lemma}[teo]{Lemma}
\theoremstyle{remark}
\theoremstyle{definition}\newtheorem{df}{Definition}[section]
\newtheorem{definition}[df]{Definition}

\renewcommand{\phi}{\varphi}

\title{Self-Avoiding Walks on the UIPQ}
\author{Alessandra Caraceni, Nicolas Curien}
\date{}
\usepackage{graphicx}
\usepackage[plainpages=false, hypertexnames=false]{hyperref}
\usepackage{bookmark}

\usepackage{cleveref}
  \crefname{thm}{Theorem}{Theorems}
    \crefname{teo}{Theorem}{Theorems}
      \crefname{theorem}{Theorem}{Theorems}
  \crefname{lem}{Lemma}{Lemmas}
  \crefname{lemma}{Lemma}{Lemmas}
  \crefname{remark}{Remark}{Remarks}
  \crefname{rem}{Remark}{Remarks}
  \crefname{proposition}{Proposition}{Propositions}
  \crefname{prop}{Proposition}{Propositions}
  \crefname{definition}{Definition}{Definitions}
  \crefname{cor}{Corollary}{Corollaries}
  \crefname{section}{Section}{Sections}
  \crefname{figure}{Figure}{Figures}

\begin{document}
\maketitle
\begin{abstract}We study an annealed model of Uniform Infinite Planar Quadrangulation (UIPQ) with an infinite two-sided self-avoiding walk (SAW), which can also be described as the result of glueing together two independent uniform infinite quadrangulations of the half-plane (UIHPQs). We prove a lower bound on the displacement of the SAW which, combined with the estimates of \cite{CMboundary}, shows that the self-avoiding walk is diffusive. As a byproduct this implies that the volume growth exponent of the lattice in question is $4$ (as is the case for the standard UIPQ); nevertheless, using our previous work \cite{CCuihpq} we show its law to be singular with respect to that of the standard UIPQ, that is  -- in the language of statistical physics -- the fact that disorder holds.  \end{abstract}
\section*{Introduction}

Much of the recent mathematical work on the geometry of random planar maps is focused on the ``pure gravity'' case where the random lattice is not affected by ``matter'': in probabilistic terms, this corresponds to choosing a map uniformly at random within a certain class, e.g.~triangulations, quadrangulations, $p$-angulations... In this work we study the geometry of random planar quadrangulations weighted by the number of their self-avoiding walks (SAWs for short); that is, we study the model of annealed SAWs on random quadrangulations. We start by presenting our main objects of interest: \medskip

\textsc{Surgeries.} The Uniform Infinite Quadrangulation of the Plane (UIPQ), denoted by $\Q_{\infty}$, is the local limit of uniform random quadrangulations whose size is sent to infinity. This object has been defined by Krikun \cite{Kri05} following the earlier work of Angel \& Schramm \cite{AS03} in the case of triangulations; since then the UIPQ has attracted a lot of attention, see \cite{Ang03,ACpercopeel,CLGplane,CMMinfini,GGN12} and references therein. 

One can also define a related object (see \cite{Ang05,CMboundary}) called the Uniform Infinite Quadrangulation of the Half-Plane with a Simple Boundary, or simple boundary UIHPQ, denoted here\footnote{Remark that this notation is not coherent with that of \cite{CCuihpq}, where we denoted by $\UIHPQ$ an object with a boundary that is not necessarily simple, and by $\UIHPQS$ the one that is central to this paper, obtained from $\UIHPQ$ by a pruning procedure. Since the general boundary UIHPQ will make no appearence in this paper, we shall drop the tilde with no fear of confusion.} by $\UIHPQ$, which is -- as the name suggests -- a random quadrangulation with an infinite simple boundary. 

The simple boundary UIHPQ can be obtained as the local limit of uniform quadrangulations with $n$ faces and a simple boundary of length $2p$ by first letting $n \to \infty$ and then $p \to \infty$ (see Section~\ref{sec:1} for more details).  From it, we shall construct two additional objects by means of ``surgery'' operations.
First, we define a random infinite quadrangulation of the plane by folding the infinite simple boundary of $\UIHPQ$ onto itself as in Figure~\ref{fig:folding}. The resulting map $\Q_{\infty}^{\rightarrow}$ is naturally endowed with an infinite one-ended self-avoiding path $(\mathsf{P}^{\rightarrow}_{i})_{i \geq 0}$ which is the image of the boundary of $\UIHPQ$. \medskip

 \begin{figure}[!h]
 \begin{center}
 \includegraphics[width=1\linewidth]{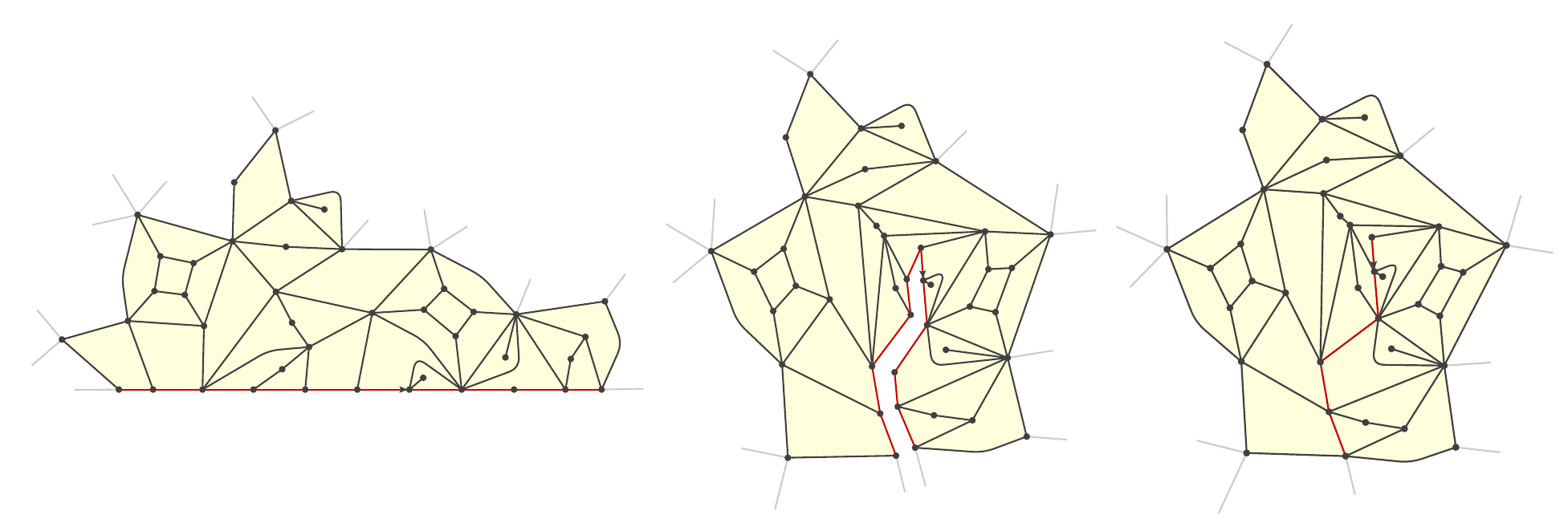}
 \caption{\label{fig:folding} A simple boundary UIHPQ and the resulting quadrangulation with a self-avoiding path obtained by folding the boundary onto itself. }
 \end{center}
 \end{figure}

We also perform a variant of the former construction. Consider two independent copies $\UIHPQ$ and  $\UIHPQ'$ of the simple boundary UIHPQ and form a quadrangulation of the plane $\Q_{\infty}^\leftrightarrow$ by glueing together $\UIHPQ$ and $\UIHPQ'$ along their boundaries (identifying the root edges with opposite orientations). This rooted infinite quadrangulation also comes with a distinguished bi-infinite self-avoiding path $({\mathsf{P}^\leftrightarrow_{i}})_{i \in \mathbb{Z}}$ resulting from the identified boundaries. 

These will be the main objects of study within this work. We will show in \cref{sec:1} that $( \Q_{\infty}^\rightarrow, {\mathsf{P}^\rightarrow})$ and $( \Q_{\infty}^\leftrightarrow, \mathsf{P}^\leftrightarrow)$ are the natural models of annealed self-avoiding walks (respectively one-sided and two-sided) on the UIPQ, which means that they can be obtained as local limits of random objects uniformly sampled among quadrangulations endowed with a self-avoiding path.  \medskip 

\textsc{Results.} According to the physics literature \cite{Dup06}, the three infinite random quadrangulations of the plane $\Q_{\infty}, \Q_{\infty}^{\rightarrow}$ and $ \Q^{\leftrightarrow}_\infty$ should be described by the same conformal field theory with central charge $c=0$; that is to say, roughly speaking, the large scale properties of $\Q_{\infty}$, $\Q_{\infty}^\rightarrow$ and $\Q_{\infty}^\leftrightarrow$ should be close to each other. We confirm this prediction by showing that these random lattices share the same volume growth exponent of $4$ (or ``Hausdorff dimension'', as it is commonly referred to by the physicists), a fact that is well-known in the case of the UIPQ, see \cite{CD06,LGM10}.  The key is to first show that the self-avoiding walks on $\Q_{\infty}^\rightarrow$ and $\Q_{\infty}^\leftrightarrow$ are  diffusive:  

\begin{theorem}[Diffusivity of the SAWs]  \label{thm:diffusive} If  $( \mathsf{P}^\rightarrow_{i})_{i \geq 0}$ and $( \mathsf{P}^\leftrightarrow_{i})_{ i \in \mathbb{Z}}$ are the edges visited by the self-avoiding walks on $ \Q_{\infty}^{\rightarrow}$ and $\Q_{\infty}^\leftrightarrow$ then we have   \begin{eqnarray} \label{eq:diffusivity} \dgr( \mathsf{P}^\rightarrow_{0}, \mathsf{P}^\rightarrow_{n})  \approx  \sqrt{n} & \mbox{ and }& \dgr( \mathsf{P}^\leftrightarrow_{0}, \mathsf{P}^\leftrightarrow_{n})  \approx    \sqrt{n}.  \end{eqnarray} 
 \end{theorem}

\noindent \textbf{Notation:} Here and later, for a random process $(X_{n})_{n \geq 0}$ with values in $ \mathbb{R}_{+}$ and a function $f : \mathbb{Z}_{+} \to \mathbb{R}_{+}$, we write $ X_{n} \preceq f(n)$ if $$ \displaystyle \lim_{a \to \infty} \limsup_{n \to \infty} \mathbb{P}( X_{n} > a f(n))=0$$ and similarly for $X_{n} \succeq f(n)$ with the reversed inequality and $a$ tending to $0$. We write $X_{n} \approx f(n)$ if we have both $X_{n} \preceq f(n)$ and $X_{n} \succeq f(n)$. \medskip 

The ball  of radius $r$ in a planar quadrangulation $  \mathfrak{q}$ is the map $ \Ball_r( \mathfrak{q})$ obtained by keeping only those (internal) faces of $ \mathfrak{q}$ that have at least one vertex at graph distance smaller than or equal to $r$ from the origin of the map (as usual all our maps are \emph{rooted}, that is given with one distinguished oriented edge whose tail vertex is the origin of the map) and keeping the root edge. The notation $\#  \Ball_{r}( \mathfrak{q})$ stands for the number of vertices in $\Ball_{r}( \mathfrak{q})$.

 \begin{cor}[Volume growth] \label{cor:volumegrowth} We have $\# \Ball_{r}( \Q^\rightarrow_{\infty}) \approx r^4$ as well as  $ \# \Ball_{r}( \Q^\leftrightarrow_{\infty}) \approx  r^4.$
 \end{cor}

Since graph distances in $\Q^\rightarrow_\infty$ and in $\Q^\leftrightarrow_{\infty}$ are trivially bounded above by distances between corresponding vertices in $\UIHPQ$ (and $\UIHPQ'$) the lower bound for the volume growth in  $\Q^\rightarrow_\infty$ and $\Q^\leftrightarrow_\infty$ follows from known results on the geometry of the UIHPQ. The nontrivial part of the statement is the upper bound, for whose proof we employ a lower bound on the displacement of the self-avoiding paths $ \mathsf{P}^\rightarrow$ and $ \mathsf{P}^{\leftrightarrow}$ (the upper bound also follows from known results on the UIHPQ \cite{CCuihpq}).  

Although $\Q_{\infty}^\leftrightarrow$ and $\Q_{\infty}$ share the same volume growth exponent, we show that their laws are very different:

\begin{theorem}[Glueing two half-planes does not produce a plane]  \label{thm:singular}The two random variables $\Q_{\infty}$ and $ \Q_{\infty}^\leftrightarrow$ are singular with respect to each other.
\end{theorem}

In the language of statistical mechanics, the former result shows that \emph{disorder} holds on the UIPQ, meaning that the (quenched) number of SAWs on $\Q_{\infty}$ is typically much less than its expectation (see Corollary~\ref{cor:fewSAW}). However, as we shall see, the proof of~Theorem \ref{thm:singular} does not involve enumerating self-avoiding walks, and is instead based on a volume argument. Unfortunately, this argument does not yield a proof of the similar result for $\Q^\rightarrow_\infty$ instead of $ \Q^{\leftrightarrow}_{\infty}$, see Conjecture~\ref{conj:singular}.

\textsc{Techniques.} In order to understand the geometry of the quadrangulations obtained from surgical operations involving the (simple boundary) UIHPQ, it is first necessary to deeply understand the geometry of the UIHPQ itself. To this end we devoted the paper \cite{CCuihpq}, in which we first considered a general boundary UIHPQ, also obtained as a local limit of uniform random quadrangulations with a boundary (on which no simplicity constraint is imposed, see \cite{CMboundary}), whose study is simpler thanks to its construction ``\`a la Schaeffer'' from a random infinite labelled tree.  The results of \cite{CCuihpq} and \cite{CMboundary} are nonetheless easily transferred to our context and yield the upper bound in Theorem \ref{thm:diffusive} and the lower bound in Corollary \ref{cor:volumegrowth}.

In order to prove the diffusive lower bound for the self-avoiding walks, the main idea is to construct disjoint paths in the UIHPQ whose endpoints lie on the boundary and are symmetric around the origin, so that after the folding of the boundary these paths become disjoint nested loops separating the origin from infinity in $\Q_{\infty}^\rightarrow$, see Figure~\ref{fig:fences one-ended} (a similar geometric construction is made in the case of $\Q_{\infty}^\leftrightarrow$). We build these paths inductively as close to each other as possible using the technique of peeling (see \cite{Ang05}) on the UIHPQ. We prove that we can construct $\approx n$ such paths on a piece of boundary of length $\approx n^2$ around the origin. Through the ``folding'' operation, this will yield the diffusivity of the self-avoiding walk $ \mathsf{P}^\rightarrow$ (resp.~$\mathsf{P}^{{\leftrightarrow}}$). Corollary~\ref{cor:volumegrowth} then follows easily by using rough bounds on the volume of balls in the UIHPQ.

The main ingredient in the proof of Theorem \ref{thm:singular} is a series of \emph{precise} estimates of the volume growth in the UIPQ and in its half-plane analogue. Indeed, the work of Le Gall \& M\'enard on the UIPQ \cite{LGM10,LGM10erratum} (see also \cite[Chapitre 4]{BouPHD}) as well as our previous work on the UIHPQ \cite{CCuihpq} show\footnote{Actually the works  \cite{CCuihpq,LGM10,LGM10erratum} show a convergence in distribution and one needs to prove uniform integrability to be able to pass to the expectation. We do not give the details since the actual proof bypasses this technical issue.} that 
 \begin{eqnarray} \label{eq:volumeexact}  \mathbb{E}[\#  \Ball_{r}(\Q_{\infty})] \sim \frac{3}{28} r^4\quad \mbox{ and } \quad  \mathbb{E}[\#  \Ball_{r}(\UIHPQ)] \sim \frac{1}{12} r^4, \mbox{ as}\quad  r \to \infty. 
 \end{eqnarray} 
As the reader will see, the constants in the above display are crucial for our purpose: since the surgeries used to create $\Q_{\infty}^\rightarrow$ and $\Q_{\infty}^\leftrightarrow$ from $ \UIHPQ $ can only decrease distances we deduce that   $$ \mathbb{E}[\# \Ball_{r}(\Q_{\infty}^\rightarrow)] \geq r^4/12 \quad \mbox{ and } \quad \mathbb{E}[\# \Ball_{r}(\Q_{\infty}^\leftrightarrow)] \geq 2 \times \frac{1}{12}r^4 = \frac{1}{6} r^{4}, \quad \mbox{ as }r \to \infty.$$ Finally, the fact that $1/6 > 3/28$ implies that balls (of large radius) around the origin in $ \Q_{\infty}^{\leftrightarrow}$ are typically larger than those in $ \Q_{\infty}$. This fact applied to different scales (so that the corresponding balls are roughly independent) is the core of the proof of Theorem \ref{thm:singular}. However, since $1/12 < 3/28$, this strategy does not work directly to prove that  the laws of $\Q_{\infty}$ and of $\Q_{\infty}^\rightarrow$ are singular with respect to each other. \medskip

{\noindent \textbf{Remark:} During the final stages of this work we became aware of the recent progresses of Gwynne and Miller \cite{GM16a,GM16b,GM16c}, who study the scaling limits in the Gromov-Hausdorff sense of the objects considered in this paper. In particular in \cite{GM16a} they prove, roughly speaking, that the glueing of random planar quadrangulations along their boundaries defines a proper glueing operation in the continuous setting after taking the scaling limit (i.e.~the image of the boundaries is a simple curve and the quotient metric does not collapse along the boundary). To do so, they use a peeling procedure which is equivalent to the one we study in Section \ref{sec:displacement}. However, the estimates provided in \cite{GM16a} are much more precise than those required and proved in this paper (their work thus greatly improves upon our Section \ref{sec:displacement}). Using the powerful theory developed by Miller \& Sheffield, the work \cite{GM16a} combined with \cite{GM16b} yields an impressive description of the Brownian surfaces glued along their boundaries in terms of $ \sqrt{8/3}$-Liouville Quantum Gravity surfaces. In particular according to \cite{GM16a}, the scaling limit of $\Q_\infty^{\leftrightarrow}$ is a weight $4$-quantum cone, the scaling limit of $\Q_\infty^{\rightarrow}$ is a weight $2$-quantum cone, whereas the Brownian plane (scaling limit of $\Q_\infty$ itself) is a weight $4/3$-quantum cone. This difference of laws in the scaling limit could probably be used instead of  \eqref{eq:volumeexact} as the main input to prove Theorem \ref{thm:singular} and could probably yield a proof of our Conjecture \ref{conj:singular}. \medskip

\noindent \textbf{Acknowledgments: } We thank J\'er\'emie Bouttier for fruitful discussion as well as for providing us with an alternative derivation of \eqref{eq:volumeexact} based on \cite{BG09}. We are also grateful to Jason Miller for a discussion about \cite{GM16a,GM16b} and Section \ref{sec:conj}. Figure 1 has been done via Timothy Budd's software.}

\tableofcontents

\section{Annealed self-avoiding walks on quadrangulations}
\label{sec:1}
We start by recalling notation and classical convergence results about random quadrangulations with a boundary. The curious reader may consult \cite{Ang03,CCuihpq,CMboundary} for details.

\subsection{Quadrangulations with a boundary}\label{sec:quad with a boundary}

Recall that all the maps we consider here are planar and \emph{rooted}, that is endowed with one distinguished oriented edge whose tail vertex 
is called the origin of the map.

 A \emph{quadrangulation with a boundary} $\q$ is a planar map all of whose faces have degree four, with the possible exception of the face lying directly \emph{to the right} of the root edge (also called the \emph{external face}, or \emph{outerface}). The external face of $\q$, whose boundary is called by extension the boundary of $\q$, necessarily has even degree (since $\q$ is bipartite); we refer to this degree as the \emph{perimeter} of $\q$, while the \emph{size} of $\q$ is the number of its faces minus 1 (so the external face is excluded). We say that $\q$ has a \emph{simple} boundary if its boundary has no pinch point, that is, if it is a cycle with no self-intersection (see Figure~\ref{fig:simple boundary finite}).
  
We denote\footnote{Notice that this is in contrast with the notation of~\cite{CCuihpq}, where a distinction needed to be made between quadrangulations with a general boundary and ones whose boundary was required to be simple, which we usually signalled with a ``tilde'' over the relevant symbol.} by $\sQ_{n,p}$ the set of all rooted quadrangulations with a simple boundary having size $n$ and perimeter $2p$ (and by $\# \sQ_{n,p}$ its cardinal). Within this paper, \emph{all quadrangulations with a boundary will be implicitly required to have a simple boundary, unless otherwise stated.}

 By convention, the set $\sQ_{0,0}$ contains a unique ``vertex'' map; more importantly, $\sQ_{0,1}$ is the set containing the unique map with one oriented edge (which has a simple boundary and no inner face). We remark that any quadrangulation with a boundary of perimeter $2$ can be seen as a rooted quadrangulation of the sphere (i.e.~without a boundary) by contracting the external face of degree two (see Figure~\ref{fig:Qn1 to Qn}); thus the set $\sQ_{n,1}$ can be identified with the set of all (rooted) quadrangulations of the sphere with $n$ faces, which we denote by $\sQ_n$. 
 
 \begin{figure}[h]\hspace{.05\textwidth}
 \begin{subfigure}[b]{.25\textwidth}\centering
\begin{tikzpicture}[scale=.6]
	\begin{pgfonlayer}{nodelayer}
		\node [style=real] (0) at (-3, -0) {};
		\node [style=real] (1) at (0, 3) {};
		\node [style=real] (2) at (3, -0) {};
		\node [style=real] (3) at (0, -3) {};
		\node [style=real] (4) at (-2.25, -2.25) {};
		\node [style=real] (5) at (-2.25, 2.25) {};
		\node [style=real] (6) at (2.25, 2.25) {};
		\node [style=real] (7) at (2.25, -2.25) {};
		\node [style=real] (8) at (-1.75, -0) {};
		\node [style=real] (9) at (-1, -1.75) {};
		\node [style=real] (10) at (0.5, 0.75) {};
		\node [style=real] (11) at (0, -1) {};
		\node [style=real] (12) at (1.25, -1) {};
		\node [style=real] (13) at (-1, 2) {};
	\end{pgfonlayer}
	\begin{pgfonlayer}{edgelayer}
		\draw[root] (4) to (3);
		\draw[very thick] (3) to (7);
		\draw[very thick] (7) to (2);
		\draw[very thick] (2) to (6);
		\draw[very thick] (6) to (1);
		\draw [bend left=75, looseness=1.50] (1) to (5);
		\draw[very thick] (5) to (0);
		\draw[very thick] (0) to (4);
		\draw (5) to (13);
		\draw[very thick] (5) to (1);
		\draw (5) to (8);
		\draw (8) to (4);
		\draw (8) to (10);
		\draw (10) to (1);
		\draw (8) to (9);
		\draw (9) to (3);
		\draw [bend left, looseness=1.00] (3) to (10);
		\draw (10) to (11);
		\draw (10) to (3);
		\draw (10) to (12);
		\draw (12) to (7);
		\draw (10) to (2);
	\end{pgfonlayer}
\end{tikzpicture}
\caption{\label{fig:simple boundary finite}} 
\end{subfigure}\hspace{.1\textwidth}
\begin{subfigure}[b]{.6\textwidth}\centering
\begin{tikzpicture}[scale=.4]
	\begin{pgfonlayer}{nodelayer}
		\node [style=real] (0) at (-4, -0) {};
		\node [style=real] (1) at (4, -0) {};
		\node [style=real] (2) at (-1.5, -0) {};
		\node [style=real] (3) at (1.5, -0) {};
		\node [style=real] (4) at (0, -0) {};
	\end{pgfonlayer}
	\begin{pgfonlayer}{edgelayer}
		\draw [very thick,bend left=45, looseness=1.00] (0) to (1);
		\draw [root,bend right=45, looseness=1.00] (0) to (1);
		\draw (0) to (2);
		\draw (3) to (1);
		\draw [bend left=75, looseness=1.50] (2) to (3);
		\draw [bend right=75, looseness=1.50] (2) to (3);
		\draw (4) to (3);
		\draw [densely dashed, ->, bend left=165, looseness=2.00] (0,2) to (0,-2);
	\end{pgfonlayer}
\end{tikzpicture}
\begin{tikzpicture}[scale=.5]
	\begin{pgfonlayer}{nodelayer}
		\node [style=real] (0) at (-3.25, 0.5) {};
		\node [style=real] (1) at (3.25, 0.5) {};
		\node [style=real] (2) at (-1.5, -0) {};
		\node [style=real] (3) at (1.5, -0) {};
		\node [style=real] (4) at (0, -0) {};
	\end{pgfonlayer}
	\begin{pgfonlayer}{edgelayer}
		\draw [root,dash pattern=on 14pt off 3pt on 3pt off 3pt,bend left=135, looseness=1.00] (0) to (1);
		\draw[black, shading=ball, ball color=white, opacity=.3] (0.02,0.6) circle (104pt);
		\draw [bend right=15, looseness=1.00] (0) to (2);
		\draw [bend right=15, looseness=1.00] (3) to (1);
		\draw [bend left=75, looseness=1.50] (2) to (3);
		\draw [bend right=75, looseness=1.50] (2) to (3);
		\draw (4) to (3);

	\end{pgfonlayer}
\end{tikzpicture}
\caption{\label{fig:Qn1 to Qn}} 
\end{subfigure}
\caption{(a) A quadrangulation in $\sQ_{9,4}$. (b) The two boundary edges of the above quadrangulation from $\sQ_{3,1}$ are ``glued together'' to obtain a rooted quadrangulation of the sphere with three faces (i.e.~an element of $\sQ_3$) on the right.} 
\end{figure}
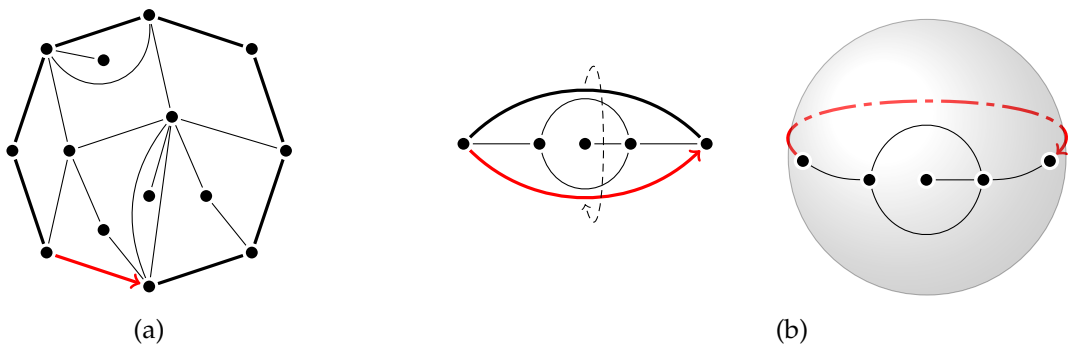

 From \cite[Eq.~(2.11)]{BG09}, we report estimates for the cardinals of the sets $\sQ_{n,p}$, where $n \geq 0, p \geq 1$:
\begin{eqnarray}
\displaystyle { \# \sQ}_{n,p} \quad = & \displaystyle 3^{-p} \frac{(3p)!}{p!(2p-1)!} 3^n\frac{(2n+p-1)!}{(n-p+1)!(n+2p)!}, & \underset{n \to \infty}{\sim} \quad {C}_{p} 12^n n^{-5/2}, \label{asympqnpt} \\
{C}_{p} \quad  = & \displaystyle \frac{1}{2 \sqrt{\pi}}  \frac{(3p)!}{p!(2p-1)!}\left(\frac{2}{3}\right)^p  & \underset{p \to \infty}{\sim} \quad   \frac{\sqrt{3p}}{2\pi}\left(\frac{9}{2}\right)^p. \label{asympcp}
\end{eqnarray} 

The sum of the series $\sum_{n\geq 0}
\# \sQ_{n,p}12^{-n}$ (which is finite) is classically denoted by $Z(p)$ and can be explicitly computed: we have $Z(1)= \frac{4}{3}$ and for $p \geq 2$,
 \begin{eqnarray}
  Z(p) \quad = \quad  2 \left( \frac{2}{3} \right)^p\frac{(3p-3)!}{p! (2p-1)!} \quad  \underset{p \to \infty}{\sim} \quad \frac{2}{9 \sqrt{3\pi}}p^{-5/2} \left(\frac{9}{2}\right)^p. 
  \end{eqnarray}
One can define a Boltzmann quadrangulation of the $2p$-gon as a random variable with values in $\bigcup_{n\geq 0} \sQ_{n,p}$, distributed according to the measure that assigns a weight
  $12^{-n} Z(p)^{-1}$ to each map in $ \sQ_{n,p}$.

  \bigskip
  In what follows, for all $n\geq 0$ and $p \geq 1$, we shall denote by $\Q_{n,p}$ a random variable uniformly distributed over $ \sQ_{n,p}$. When $p=1$ we also denote $\Q_{n}:=\Q_{n,1}$ a uniform quadrangulation with $n$ faces.

\subsection{Uniform Infinite (Half-)Planar Quadrangulations}

Recall that if $ \mathfrak{q}, \mathfrak{q}'$ are two rooted (planar) quadrangulations (with or without a boundary), the local distance between the two is
 \begin{eqnarray} \label{eq:deflocal} \mathrm{d_{loc}}( \mathfrak{q}, \mathfrak{q}') &=& \big(1+ \sup\{r\geq 0 : \Ball_r( \mathfrak{q})=  \Ball_r( \mathfrak{q}')\}\big)^{-1}, \end{eqnarray}
where $\Ball_r( \mathfrak{q})$ is obtained by erasing from $ \mathfrak{q}$ everything but those inner faces that have at least one vertex at distance smaller than or equal to $r$ from the origin (thus the outerface is not automatically preserved if $ \mathfrak{q}$ has a boundary). The set of all finite quadrangulations with a boundary is not complete for this metric: we shall work in its completion, obtained by adding locally finite infinite quadrangulations with a finite or infinite simple boundary, see \cite{CMMinfini} for details. Recall that $\Q_{n,p}$ is uniformly distributed over $ \sQ_{n,p}$. The following convergences in distribution for $ \mathrm{d_{loc}}$ are by now well known:
\begin{eqnarray} \Q_{n,p}  \quad \xrightarrow[n\to \infty]{(d)} \quad  \Q_{\infty,p} \quad \xrightarrow[p\to\infty]{(d)} \quad \UIHPQ.  \label{def:UIPQ} \label{def:uipq2pgon} \end{eqnarray}
The first convergence in the special case $p=1$ constitutes the definition of the UIPQ by Krikun \cite{Kri05}; the second one is found in  \cite{CMboundary} (see also the pioneering work \cite{Ang05} concerning the triangulation case). The object $\Q_{\infty,p}$ is the so-called UIPQ of the $2p$-gon and  $\UIHPQ$ is the simple boundary UIHPQ. 

It is worthwhile to note (such a fact will be useful later) that $\UIHPQ$ enjoys a property of invariance under rerooting: if we shift the root edge by one along the boundary (to the left or right), the random map thus obtained still has the law of a (simple boundary) UIHPQ.

\subsection{Zipper}

Let us now give a precise definition of a self-avoiding path:

\begin{definition}
Let $ \mathfrak{q}$ be a (finite or infinite) planar quadrangulation, and let $b  \in \{0,1, \ldots \} \cup \{\infty\}$, $f \in \{1,2, \ldots \} \cup \{ \infty\}$ (``$b$'' stands for backward and ``$f$'' for forward). A $(b,f)$-SAW  on $ \q$ is a sequence $$  \mathbf{w}=(\vec{e}_{i})_{-b\leq i<f}$$
of successive oriented edges of the map, where $\vec{e}_{i}$ has tail vertex $x_i$ and target vertex $x_{i+1}$, so that the target of $\vec{e}_i$ coincides with the tail of $\vec{e}_{i+1}$, the oriented edge $\vec{e}_0$ is the root of $\q$ (thus $x_0$ the origin) and the vertices in the sequence $(x_i)_{-b\leq i\leq f}$ are distinct, see Figure~\ref{fig:4,3-SAW}.
\end{definition}

 We shall call $\sQ_n^{b,f}$, the set of all pairs $(\q,P)$, where $\q\in\sQ_n$ and $P$ is a $(b,f)$-SAW on $\q$ (so that the set $\sQ_n^{0,1}$ is automatically identified with the set $\sQ_n$).
 
 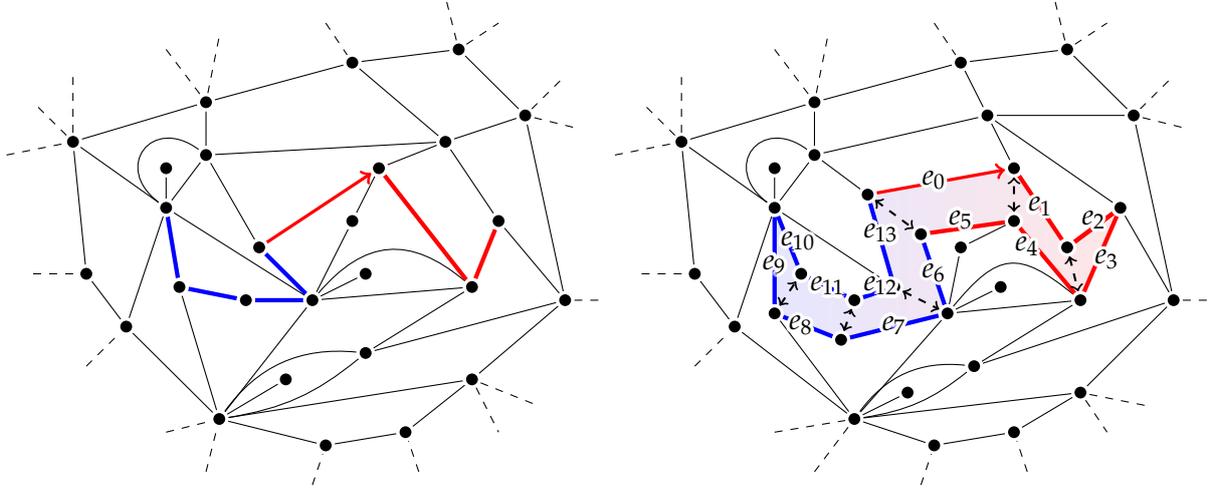
\begin{figure}[h!]\centering
\begin{tikzpicture}[b/.style={blue, ultra thick},f/.style={red, ultra thick}, scale=.70]
	\begin{pgfonlayer}{nodelayer}
		\node [style=real] (0) at (-1, -0) {};
		\node [style=real] (1) at (1.25, 1.5) {};
		\node [style=real] (2) at (0.75, 0.5) {};
		\node [style=real] (3) at (0, -1) {};
		\node [style=real] (4) at (3.5, 0.5) {};
		\node [style=real] (5) at (2.5, 2) {};
		\node [style=real] (6) at (1, -0.5) {};
		\node [style=real] (7) at (3, -0.75) {};
		\node [style=real] (8) at (-2.75, 0.75) {};
		\node [style=real] (9) at (-2, 2.75) {};
		\node [style=real] (10) at (-2, 1.75) {};
		\node [style=real] (11) at (-4.25, -0.5) {};
		\node [style=real] (12) at (-3.5, -1.5) {};
		\node [style=real] (13) at (-2.5, -0.75) {};
		\node [style=real] (14) at (-1.75, -3.25) {};
		\node [style=real] (15) at (-1.25, -1) {};
		\node [style=real] (16) at (1, -2) {};
		\node [style=real] (17) at (3, -2.5) {};
		\node [style=real] (18) at (4.75, -1) {};
		\node [style=real] (19) at (0.25, -3.75) {};
		\node [style=real] (20) at (1.75, -3.5) {};
		\node [style=real] (21) at (-2.75, 1.5) {};
		\node [style=real] (22) at (-4.5, 2) {};
		\node [style=real] (23) at (2.75, 3.75) {};
		\node [style=real] (24) at (0.75, 3.5) {};
		\node [style=real] (25) at (4, 2.5) {};
		\node [style=real] (26) at (-0.5, -2.5) {};
		\node [style=none] (27) at (-2.75, 3.75) {};
		\node [style=none] (28) at (-1.75, 4) {};
		\node [style=none] (29) at (0.25, 4.25) {};
		\node [style=none] (30) at (2.5, 4.75) {};
		\node [style=none] (31) at (3.5, 4.25) {};
		\node [style=none] (32) at (4.75, 3.25) {};
		\node [style=none] (33) at (5, 2.25) {};
		\node [style=none] (34) at (5.5, -1) {};
		\node [style=none] (35) at (3.5, -3.5) {};
		\node [style=none] (36) at (4.25, -3) {};
		\node [style=none] (37) at (2, -4.25) {};
		\node [style=none] (38) at (0, -4.5) {};
		\node [style=none] (39) at (-2.75, -3.5) {};
		\node [style=none] (40) at (-2, -4.25) {};
		\node [style=none] (41) at (-4.25, -2.25) {};
		\node [style=none] (42) at (-5.25, -0.5) {};
		\node [style=none] (43) at (-5.75, 1.75) {};
		\node [style=none] (44) at (-5.25, 2.75) {};
		\node [style=none] (45) at (-4.5, 3.25) {};
	\end{pgfonlayer}
	\begin{pgfonlayer}{edgelayer}
		\draw [style=root] (0) to (1);
		\draw (1) to (2);
		\draw (2) to (3);
		\draw [style=b] (3) to (0);
		\draw [bend left=45, looseness=1.25] (3) to (7);
		\draw [style=f] (7) to (1);
		\draw [style=f] (7) to (4);
		\draw (4) to (5);
		\draw (5) to (1);
		\draw (0) to (10);
		\draw (10) to (5);
		\draw (10) to (9);
		\draw (9) to (24);
		\draw (24) to (5);
		\draw [bend right=90, looseness=2.25] (10) to (8);
		\draw (3) to (8);
		\draw [style=b] (8) to (13);
		\draw [style=b] (13) to (15);
		\draw [style=b] (15) to (3);
		\draw (13) to (14);
		\draw (14) to (3);
		\draw [bend right, looseness=1.00] (16) to (14);
		\draw (16) to (7);
		\draw (16) to (18);
		\draw (18) to (4);
		\draw (5) to (25);
		\draw (25) to (18);
		\draw (18) to (17);
		\draw (17) to (14);
		\draw (14) to (19);
		\draw (19) to (20);
		\draw (20) to (17);
		\draw (14) to (12);
		\draw (12) to (8);
		\draw (8) to (22);
		\draw (9) to (22);
		\draw (22) to (11);
		\draw (11) to (12);
		\draw (24) to (23);
		\draw (23) to (25);
		\draw (3) to (6);
		\draw (21) to (8);
		\draw (3) to (7);
		\draw (8) to (10);
		\draw [bend right=15, looseness=1.00] (14) to (16);
		\draw (26) to (14);
		\draw [style=dashed] (27.center) to (9);
		\draw [style=dashed] (9) to (28.center);
		\draw [style=dashed] (29.center) to (24);
		\draw [style=dashed] (23) to (30.center);
		\draw [style=dashed] (23) to (31.center);
		\draw [style=dashed] (25) to (32.center);
		\draw [style=dashed] (25) to (33.center);
		\draw [style=dashed] (18) to (34.center);
		\draw [style=dashed] (17) to (35.center);
		\draw [style=dashed] (17) to (36.center);
		\draw [style=dashed] (37.center) to (20);
		\draw [style=dashed] (38.center) to (19);
		\draw [style=dashed] (40.center) to (14);
		\draw [style=dashed] (39.center) to (14);
		\draw [style=dashed] (41.center) to (12);
		\draw [style=dashed] (42.center) to (11);
		\draw [style=dashed] (43.center) to (22);
		\draw [style=dashed] (22) to (44.center);
		\draw [style=dashed] (22) to (45.center);
	\end{pgfonlayer}\end{tikzpicture}	\begin{tikzpicture}[b/.style={blue, ultra thick},f/.style={red, ultra thick}, scale=.70, simple/.style={}]
	\begin{pgfonlayer}{nodelayer}
		\node [style=real] (0) at (-1, 1) {};
		\node [style=real] (1) at (1.75, 1.5) {};
		\node [style=real] (2) at (0.75, -0) {};
		\node [style=real] (3) at (-0.5, -0.75) {};
		\node [style=real] (4) at (3.75, 0.75) {};
		\node [style=real] (5) at (1.25, 2.5) {};
		\node [style=real] (6) at (1.5, -0.75) {};
		\node [style=real] (7) at (2.75, -0) {};
		\node [style=real] (8) at (-2.75, 0.75) {};
		\node [style=real] (9) at (-2, 2.75) {};
		\node [style=real] (10) at (-2, 1.75) {};
		\node [style=real] (11) at (-4.25, -0.5) {};
		\node [style=real] (12) at (-3.5, -1.5) {};
		\node [style=real] (13) at (-2.25, -0.5) {};
		\node [style=real] (14) at (-1.25, -3.25) {};
		\node [style=real] (15) at (-1.25, -1) {};
		\node [style=real] (16) at (1, -2.25) {};
		\node [style=real] (17) at (3, -2.75) {};
		\node [style=real] (18) at (4.75, -1) {};
		\node [style=real] (19) at (0.25, -3.75) {};
		\node [style=real] (20) at (1.75, -3.5) {};
		\node [style=real] (21) at (-2.75, 1.5) {};
		\node [style=real] (22) at (-4.5, 2) {};
		\node [style=real] (23) at (2.75, 3.75) {};
		\node [style=real] (24) at (0.75, 3.5) {};
		\node [style=real] (25) at (4, 2.5) {};
		\node [style=real] (26) at (-0.25, -2.75) {};
		\node [style=none] (27) at (-2.75, 3.75) {};
		\node [style=none] (28) at (-1.75, 4) {};
		\node [style=none] (29) at (0.25, 4.25) {};
		\node [style=none] (30) at (2.5, 4.75) {};
		\node [style=none] (31) at (3.5, 4.25) {};
		\node [style=none] (32) at (4.75, 3.25) {};
		\node [style=none] (33) at (5, 2.25) {};
		\node [style=none] (34) at (5.5, -1) {};
		\node [style=none] (35) at (3.5, -3.5) {};
		\node [style=none] (36) at (4.25, -3) {};
		\node [style=none] (37) at (2, -4.25) {};
		\node [style=none] (38) at (0, -4.5) {};
		\node [style=none] (39) at (-2.75, -3.5) {};
		\node [style=none] (40) at (-2, -4.25) {};
		\node [style=none] (41) at (-4.25, -2.25) {};
		\node [style=none] (42) at (-5.25, -0.5) {};
		\node [style=none] (43) at (-5.75, 1.75) {};
		\node [style=none] (44) at (-5.25, 2.75) {};
		\node [style=none] (45) at (-4.5, 3.25) {};
		\node [style=real] (46) at (-2.75, -1.25) {};
		\node [style=real] (47) at (-1.5, -1.75) {};
		\node [style=real] (48) at (0.5, -1.25) {};
		\node [style=real] (49) at (0, 0.25) {};
		\node [style=real] (50) at (1.75, 0.5) {};
		\node [style=real] (51) at (3, -1) {};
		
		\contourlength{1.2pt}
		\node [style=none] (52) at (0.25, 1.25) {\contour{white}{$e_0$}};
		\node [style=none] (53) at (2.25, 0.75) {\contour{white}{$e_1$}};
				\node [style=none] (54) at (3.25, 0.5) {\contour{white}{$e_2$}};
				\node [style=none] (65) at (3.5, -0.25) {\contour{white}{$e_3$}};
		\node [style=none] (55) at (2, -0) {\contour{white}{$e_4$}};
		\node [style=none] (56) at (0.75, 0.5) {\contour{white}{$e_5$}};
		\node [style=none] (57) at (0.25, -0.5) {\contour{white}{$e_6$}};
		\node [style=none] (58) at (-0.5, -1.5) {\contour{white}{$e_7$}};
		\node [style=none] (59) at (-2.25, -1.5) {\contour{white}{$e_8$}};
		\node [style=none] (60) at (-2.75, -0.35) {\contour{white}{$e_9$}};
		\node [style=none] (61) at (-2.3, 0.1) {\contour{white}{$e_{10}$}};
		\node [style=none] (62) at (-1.75, -0.75) {\contour{white}{$e_{11}$}};
		\node [style=none] (63) at (-0.75, -0.75) {\contour{white}{$e_{12}$}};
		\node [style=none] (64) at (-0.75, 0.25) {\contour{white}{$e_{13}$}};
	\end{pgfonlayer}
	\begin{pgfonlayer}{edgelayer}
	\fill[left color=blue!10, right color=red!10]
	 (8.center)--(13.center)--(15.center)--(3.center)--(0.center)--(1.center)--(7.center)--(4.center)--(51.center)--(50.center)--(49.center)--(48.center)--(47.center)--(46.center)--(8.center);
	
		\draw [style=root] (0) to (1);
		\draw [style=b] (3) to (0);
		\draw [style=f] (7) to (1);
		\draw [style=f] (7) to (4);
		\draw (4) to (5);
		\draw (5) to (1);
		\draw (0) to (10);
		\draw (10) to (5);
		\draw (10) to (9);
		\draw (9) to (24);
		\draw (24) to (5);
		\draw [bend right=90, looseness=2.25] (10) to (8);
		\draw (3) to (8);
		\draw [style=b] (8) to (13);
		\draw [style=b] (13) to (15);
		\draw [style=b] (15) to (3);
		\draw [bend right, looseness=1.00] (16) to (14);
		\draw (16) to (18);
		\draw (18) to (4);
		\draw (5) to (25);
		\draw (25) to (18);
		\draw (18) to (17);
		\draw (17) to (14);
		\draw (14) to (19);
		\draw (19) to (20);
		\draw (20) to (17);
		\draw (14) to (12);
		\draw (12) to (8);
		\draw (8) to (22);
		\draw (9) to (22);
		\draw (22) to (11);
		\draw (11) to (12);
		\draw (24) to (23);
		\draw (23) to (25);
		\draw (21) to (8);
		\draw (8) to (10);
		\draw [bend right=15, looseness=1.00] (14) to (16);
		\draw (26) to (14);
		\draw [style=dashed] (27.center) to (9);
		\draw [style=dashed] (9) to (28.center);
		\draw [style=dashed] (29.center) to (24);
		\draw [style=dashed] (23) to (30.center);
		\draw [style=dashed] (23) to (31.center);
		\draw [style=dashed] (25) to (32.center);
		\draw [style=dashed] (25) to (33.center);
		\draw [style=dashed] (18) to (34.center);
		\draw [style=dashed] (17) to (35.center);
		\draw [style=dashed] (17) to (36.center);
		\draw [style=dashed] (37.center) to (20);
		\draw [style=dashed] (38.center) to (19);
		\draw [style=dashed] (40.center) to (14);
		\draw [style=dashed] (39.center) to (14);
		\draw [style=dashed] (41.center) to (12);
		\draw [style=dashed] (42.center) to (11);
		\draw [style=dashed] (43.center) to (22);
		\draw [style=dashed] (22) to (44.center);
		\draw [style=dashed] (22) to (45.center);
		\draw (46) to (14);
		\draw (48) to (14);
		\draw (16) to (51);
		\draw [style=b] (8) to (46);
		\draw [style=b] (46) to (47);
		\draw [style=b] (47) to (48);
		\draw [style=b] (48) to (49);
		\draw [style=f] (49) to (50);
		\draw [style=f] (7) to (4);
		\draw [style=simple, bend left=45, looseness=1.50] (48) to (51);
		\draw [style=simple] (48) to (2);
		\draw [style=simple] (2) to (50);
		\draw [style=simple] (48) to (6);
		\draw [style=simple] (48) to (51);
		\draw [style=f] (50) to (51);
		\draw [style=f] (51) to (4);
			
		\draw [style=dashed,<->, thick] (46) to (13);
		\draw [style=dashed,<->, thick] (47) to (15);
		\draw [style=dashed,<->, thick] (3) to (48);
		\draw [style=dashed,<->, thick] (0) to (49);
		\draw [style=dashed,<->, thick] (50) to (1);
		\draw [style=dashed,<->, thick] (7) to (51);
	\end{pgfonlayer}
\end{tikzpicture}

\caption{\label{fig:4,3-SAW}Left: A $(4,3)$-SAW in a rooted quadrangulation. Right: The function $\Z{4}{3}$ is applied to a quadrangulation with a boundary of length 14 (whose outerface is drawn as finite to aid visualisation) to obtain the quadrangulation with a distinguished $(4,3)$-SAW depicted on the Left. \label{fig:zip43} }
\end{figure}

Fix $p \geq 1$. There is an obvious bijective correspondence between, on the one hand, the set $\sQ^{0,p}_n$ of quadrangulations of size $n$ with a $(0,p)$-SAW and, on the other hand, the set $\sQ_{n,p}$ of quadrangulations with a simple boundary of perimeter $2p$ and size $n$; such a correspondence is an immediate generalisation of the one between $\sQ_{n,1}$ and $\sQ_n$ mentioned in Section~\ref{sec:quad with a boundary} (Figure~\ref{fig:Qn1 to Qn}): simply let the self-avoiding walk act as a ``zipper'', eliminating the external face by pairwise identifying its edges.

 In fact, we may generalize this construction further: for $b \geq 0, f \geq 1$ such that $b+f=p$ one can build a bijection $\Z{b}{f}$ between the set of all finite quadrangulations with a simple boundary of length $2p$ and the set of all finite quadrangulations of the sphere endowed with a $(b,f)$-SAW. Such a mapping works as follows: write $\vec{e}_0,\ldots , \vec{e}_{2p-1}$ for the $2p$ edges of the boundary of a quadrangulation $ \mathfrak{q}\in \sQ_{n,p}$, taken in clockwise order and in such a way that $\vec{e}_0$ is the root edge, each edge oriented clockwise with respect to the outerface. We set $ \Z{b}{f}( \mathfrak{q})$ to be the quadrangulation of the sphere obtained by identifying $\vec{e}_i$ with $-\vec{e}_{2f-1-i}$ (where indices are to be read modulo $2p$ and the minus sign represents a change in orientation), endowed with the distinguished self-avoiding path of length $p$ that is the image of the original cycle $\vec{e}_0,\ldots \vec{e}_{2p-1}$ and rooted at the image of $\vec{e}_0$, see Figure \ref{fig:zip43}.

 Since the above mappings are bijections, if $\Q_{n,p}$ is uniformly distributed over $ \sQ_{n,p}$ then for any fixed quadrangulation of the sphere $\q$ with $n$ faces we have 
 \begin{eqnarray} \label{eq:density1} \mathbb{P}\big( \Z{b}{f}(\Q_{n,p}) = (\mathfrak{q}, \mathbf{w})  \mbox{ for some $(b,f)$-SAW } \mathbf{w}\big) =  \frac{\# \SAW{b}{f}( \mathfrak{q})}{ \# \sQ_{n,b+f}},  \end{eqnarray}
 where $ \# \SAW{b}{f}( \mathfrak{q})$ is the number of $(b,f)$-SAWs on $ \q$. In other words, the underlying quadrangulation of $\Z{b}{f}(\Q_{n,p})$ is \emph{not} uniformly distributed, but biased by its number of $(b,f)$-SAWs.

 \subsection{Annealed infinite self-avoiding walks on the UIPQ}

   One can extend the definition of the local distance to maps endowed with a distinguished SAW as a variant of \eqref{eq:deflocal}, by providing an appropriate notion of a ball: if $( \mathfrak{q}, ( \vec{e}_i)_{-b-1<  i < f })$ is a quadrangulation with a distinguished SAW of type $(b,f)$, for each $r \geq 2$ we set $$\Ball_{r}( \mathfrak{q}, ( \vec{e}_i)_{-b-1 <  i < f }) = \Big( \Ball_{r}( \mathfrak{q}),( \vec{e}_i)_{-(b\wedge (r-1))-1 <  i <  (f\wedge (r-1)) }\Big).$$

 For any fixed $b,f$, it is clear that the zipper map $\Z{b}{f}$ is continuous for the local topology, hence one may deduce from \eqref{def:UIPQ} that for any $b\geq 0$, $f \geq 1$ such that $b+f=p$ one has $ \Z{b}{f}(\Q_{n,p}) \to \Z{b}{f}(\Q_{\infty,p})$ in distribution as $n \to \infty$. We are now interested in letting $b$ and $f$ tend to $\infty$.
   \begin{proposition}[Annealed UIPQs with SAW] \label{prop:annealedSAW} We have the following convergences in distribution for the local topology on quadrangulations endowed with a self-avoiding walk:
 \begin{eqnarray} \label{eq:defUIPQSAW1} \Z{0}{p}(\Q_{\infty,p}) &\xrightarrow[p\to\infty]{(d)} &(\Q_{\infty}^\rightarrow, (\mathsf{P}^\rightarrow_i)_{i \geq 0}), \\
 \label{eq:defUIPQSAW2} \Z{p}{p'}(\Q_{\infty,p+p'}) &\xrightarrow[p,p'\to\infty]{(d)}& (\Q_{\infty}^\leftrightarrow, (\mathsf{P}^\leftrightarrow_i)_{i \in\mathbb{Z}}),  \end{eqnarray}
 where $(\Q_{\infty}^\rightarrow, (\mathsf{P}^\rightarrow_i)_{i \geq 0})$ can be obtained as $ \Z{0}{\infty}(\UIHPQ)$ by ``zipping up'' the boundary of a UIHPQ, whereas $(\Q_{\infty}^\leftrightarrow, (\mathsf{P}^\leftrightarrow_i)_{i \in \mathbb{Z}})$ is the result of the glueing of two independent UIHPQs along their boundaries (so that their root edges are identified with opposite orientations).
\end{proposition}
   
\proof Consider the first convergence \eqref{eq:defUIPQSAW1}; we claim that it is a consequence of the second convergence in \eqref{def:UIPQ}. To see this, notice first that we can extend the definition of the zipper map and consider $ \Z{b}{f}$ when one out of $b,f$ is finite, the other infinite;  such a correspondence maps an infinite quadrangulation with an infinite boundary to an infinite quadrangulation endowed with a $(b,f)$-SAW, as depicted in Figure~\ref{fig:(2,infinity)-SAW}. 

\begin{figure}[!h]
 \begin{center}
 \includegraphics[width=13cm]{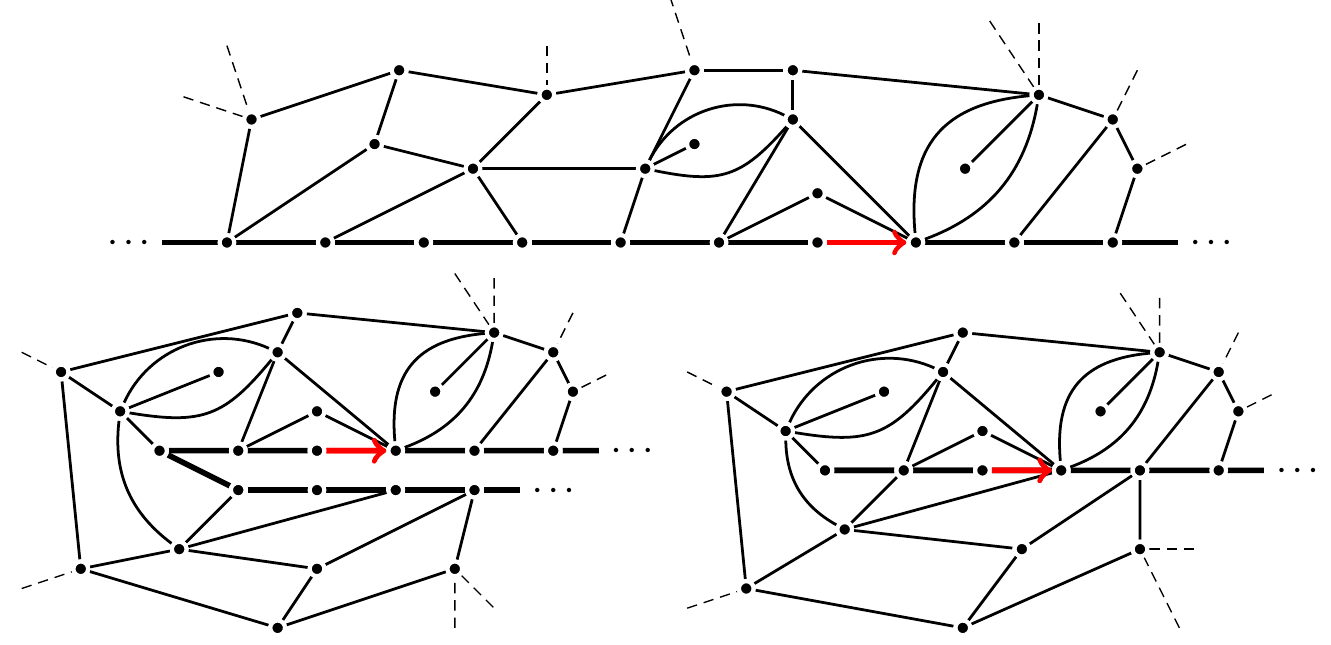}
 \caption{\label{fig:(2,infinity)-SAW}The map $\Z{2}{\infty}$ applied to a quadrangulation with an infinite boundary.}
 \end{center}
 \end{figure}
\begin{lemma} Let $ \mathfrak{q}_p \to  \mathfrak{q}_\infty$ be a sequence of quadrangulations with a boundary, $\q_p$ having perimeter $2p$, which converges for $ \mathrm{d_{loc}}$ towards an infinite quadrangulation with an infinite boundary; then we have 
 \begin{eqnarray*} \Z{0}{p}( \mathfrak{q}_p) &\xrightarrow[p\to\infty]{ ( \mathrm{d_{loc}})}& \Z{0}{\infty}( \mathfrak{q}_\infty).  \end{eqnarray*}
\end{lemma} 
\proof[Proof of the lemma] Fix $r \geq 1$. Although $\Ball_{r}( \Z{0}{\infty}( \mathfrak{q}_\infty))$  may not be a measurable function of $ \Ball_{r}(\mathfrak{q}_\infty)$ (because graph distances may be decreased by applying $\Z{0}{\infty}$), it is easy to see that one can find $r' \geq r$ (depending on $ \mathfrak{q}_\infty$) such that if $\Ball_{r'}( \mathfrak{q}_\infty) =  \Ball_{r'}( \mathfrak{q}_p)$ then we have $ \Ball_{r}(\Z{0}{p}( \mathfrak{q}_p)) = \Ball_{r}( \Z{0}{\infty}( \mathfrak{q}_\infty))$. This proves the lemma.
\endproof
Coming back to the proof of the theorem, by \eqref{def:UIPQ} and the  Skorokhod embedding theorem one can suppose that $\Q_{\infty,p} \to \UIHPQ$ almost surely. It thus follows from the above lemma that $ \Z{0}{p}( \Q_{\infty,p}) \to \Z{0}{\infty}( \UIHPQ)$ almost surely as $p\to\infty$. This proves the desired convergence in distribution.\\

We now move on to the second convergence \eqref{eq:defUIPQSAW2}, which is not this time a simple consequence of \eqref{def:UIPQ}, as one cannot define $ \Z{\infty}{\infty}( \UIHPQ)$. The idea is that the two parts of $\Q_{\infty,p+p'}$ which are facing together near the root edge in $ \Z{p}{p'}(\Q_{\infty,p+p'})$ are distant from each other when $p,p' \to \infty$ and become asymptotically independent. Here is the proper lemma from which the second convergence \eqref{eq:defUIPQSAW2} immediately follows:

\begin{lemma}  \label{lem:indept}For $k \in \{0, \ldots, 2p\}$ denote by $\Q_{\infty,p}^{(k)}$ the random infinite quadrangulation with a boundary of perimeter $2p$ obtained by re-rooting $\Q_{\infty,p}$ at the $k$-th edge along the boundary of its external face. Then we have 
$$ (\Q_{\infty,p}, \Q_{\infty,p}^{(k)}) \xrightarrow[\begin{subarray}{c} k\to\infty \\
(2p-k) \to \infty \end{subarray}]{} (\UIHPQ, \UIHPQ'),$$ where $\UIHPQ$ and $\UIHPQ'$ are two independent copies of the UIHPQ.
\end{lemma}
\proof[Proof of the lemma] Notice first that by invariance under re-rooting $\Q_{\infty,p}$ and $ \Q_{\infty,p}^{(k)}$ have the same law and both converge in law towards $\UIHPQ$ by \eqref{def:UIPQ}. The only nontrivial point is the asymptotic independence. Let $r \geq 1$; we will show that the $r$-neighborhoods around the root edge and the $k$-th edge along the boundary of $\Q_{\infty,p}$ become independent as $k \to \infty$ and $2p-k \to \infty$. We write $  \Ball_r^\bullet(\Q_{\infty,p})$ for the hull of the ball of radius $r$ inside $\Q_{\infty,p}$: this is the submap obtained by filling in all the finite holes that $ \Ball_{r}(\Q_{\infty,p})$ together with the boundary of $\Q_{\infty,p}$ may create (recall that $\Q_{\infty,p}$ only has one end), see Figure~\ref{fig:ball}. Hence $ \Ball_r^\bullet(\Q_{\infty,p})$ is a finite quadrangulation with a simple boundary made up of two joined paths: one belonging to the boundary of $\Q_{\infty,p}$ (the outer boundary) and the other (the inner boundary) on which one needs to glue an infinite quadrangulation with a boundary in order to recover $\Q_{\infty,p}$ (see \cite[Section 4.1]{CLGmodif} for a similar definition in the context of triangulations). 

\begin{figure}[!h]
\centering
\begin{tikzpicture}[b/.style={line width=3pt, white}]
	\begin{pgfonlayer}{nodelayer}
		\node [style=real, fill=red, label=below:0] (0) at (0, -0) {};
		\node [style=real, fill=red, label=below:1] (1) at (1, -0) {};
		\node [style=real, fill=red, label=below right:1] (12) at (1, 1.75) {};
		\node [style=real, fill=red, label=above:1] (13) at (-1.5, 1.75) {};
		\node [style=real, fill=red, label=below:1] (6) at (-1, -0) {};
		\node [style=real, fill=blue!60, label=below:2] (2) at (2, -0) {};
		\node [style=real, fill=blue!60, label=below:2] (4) at (4, -0) {};
		\node [style=real, fill=blue!60, label=above:2] (14) at (0, 3.25) {};
		\node [style=real, fill=blue!60, label=below:2] (7) at (-2, -0) {};
		\node [style=real, fill=blue!60, label=below:2] (9) at (-4, -0) {};
		\node [style=real, fill=blue!60, label=below:2] (15) at (0, 1.25) {};
		\node [style=real, fill=blue!80, label=below:3] (3) at (3, -0) {};
		\node [style=real, fill=blue!80, label=below:3] (5) at (5, -0) {};
		\node [style=real, fill=blue!80, label=below:3] (8) at (-3, -0) {};
		
		\node [style=real, fill=blue!80, label=below:3] (10) at (-5, -0) {};
		\node [style=real] (11) at (2.5, 0.25) {};
		
		\node [style=real, fill=blue!80, label=left:3] (16) at (-0.5, 2) {};
		\node [style=real, fill=blue!80, label=right:3] (17) at (0.5, 2.5) {};
		\node [style=real, fill=blue!80, label=below:3] (18) at (-2.75, 2) {};
		\node [style=real, fill=blue, label=above:\contour{white}{4}] (11) at (2.5, 0.25) {};

		\node (19) at (1, 3.75) {};
		\node (20) at (5.5, 0.75) {};
		\node (21) at (-6, 0.75) {};
	\end{pgfonlayer}
	\begin{pgfonlayer}{edgelayer}
	\fill[red!20] (9.center)--(5.center) [bend right=45, looseness=1.05] to (14.center)--(18.center)--(9.center);
	\fill[gray!40!red] (14.center)--(16.center)--(15.center)--(17.center)--cycle;
	\fill[gray!40!red][bend left=90, looseness=2.2] (2.center) to (3.center)--(2.center);
		\draw[b] [bend left=90, looseness=2.00] (2) to (3);
		\draw[b] (3) to (11);
		\draw[b] (2) to (3);
		\draw[b] [bend left=75, looseness=1.25] (1) to (4);
		\draw[b] (1) to (2);
		\draw[b] (3) to (4);
		\draw[b] (13) to (0);
		\draw[b] (0) to (12);
		\draw[b] [bend left=45, looseness=1.00] (12) to (4);
		\draw[b] (0) to (1);
		\draw[b] (13) to (15);
		\draw[b] (15) to (12);
		\draw[b] (13) to (14);
		\draw[b] [bend left=60, looseness=1.50] (14) to (12);
		\draw[b] (14) to (17);
		\draw[b] (17) to (15);
		\draw[b] (16) to (15);
		\draw[b] (16) to (14);
		\draw[b] (13) to (7);
		\draw[b] (7) to (6);
		\draw [b](6) to (0);
		\draw[b] (9) to (8);
		\draw[b] (8) to (7);
		\draw[b] (9) to (13);
		\draw[b] (14) to (18);
		\draw[b] (18) to (9);
		\draw[b] (10) to (9);
		\draw[b] [bend left=45, looseness=1.25] (10) to (14);
		\draw[b] (4) to (5);
		\draw[b] [bend left=45, looseness=1.00] (14) to (5);
		
		\draw[dashed] (-6,0)--(6,0);
		\draw[thick] (-5,0)--(5,0);
		\draw[ultra thick] (-4,0)--(5,0);
		\draw [bend left=90, looseness=2.00] (2) to (3);
		\draw (3) to (11);
		\draw (2) to (3);
		\draw [bend left=75, looseness=1.25] (1) to (4);
		\draw (1) to (2);
		\draw (3) to (4);
		\draw (13) to (0);
		\draw (0) to (12);
		\draw [bend left=45, looseness=1.00] (12) to (4);
		\draw[root] (0) to (1);
		\draw (13) to (15);
		\draw (15) to (12);
		\draw (13) to (14);
		\draw [bend left=60, looseness=1.50] (14) to (12);
		\draw (14) to (17);
		\draw (17) to (15);
		\draw (16) to (15);
		\draw (16) to (14);
		\draw (13) to (7);
		\draw (7) to (6);
		\draw (6) to (0);
		\draw (9) to (8);
		\draw (8) to (7);
		\draw (9) to (13);
		\draw[ultra thick, red] (14) to (18);
		\draw[ultra thick, red] (18) to (9);
		\draw (10) to (9);
		\draw [bend left=45, looseness=1.25] (10) to (14);
		\draw (4) to (5);
		\draw [ultra thick, red, bend left=45, looseness=1.00] (14) to (5);
	
			\draw[dashed] (14) to (19.center);
		\draw[dashed] (20.center) to (5);
		\draw[dashed] (10) to (21.center);
	\end{pgfonlayer}
\end{tikzpicture}
 \caption{\label{fig:ball}The hull of the ball of radius $1$ inside an infinite quadrangulation of the half-plane. The outer boundary is the thick black line and the inner boundary is in red.}

 \end{figure}
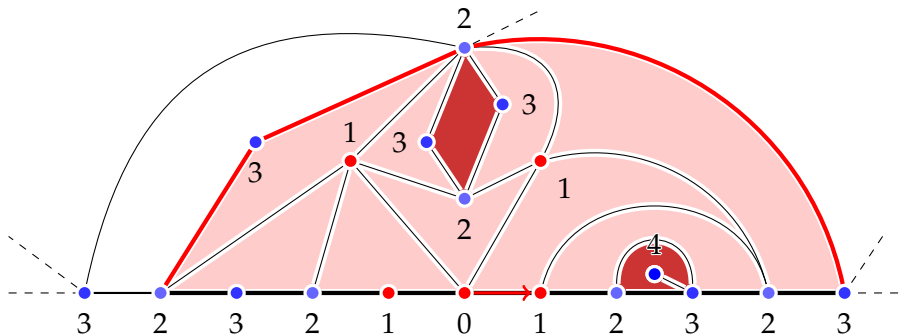

The spatial Markov property of the UIPQ of the $2p$-gon (see \cite{Ang03, CMboundary,CLGpeeling}) shows that conditionally on $ \{\Ball_r^\bullet( \Q_{\infty,p}) = \mathfrak{a}\}$ the law of the remaining part of $\Q_{\infty,p}$ is that of of a UIPQ of the $(2p + \ell_{ \mathrm{in}} - \ell_{ \mathrm{ out}})$-gon where $\ell_{ \mathrm{in}}$ and $\ell_{ \mathrm{out}}$ are respectively the length of the inner and outer boundary of $ \mathfrak{a}$. Now if $k \to \infty$ and $2p-k \to \infty$, by local finiteness, it is very unlikely that $ \Ball_{r}^{\bullet}(\Q_{\infty,p})$ intersects $ \Ball_{r}^{\bullet}( \Q_{\infty,p}^{{(k)}})$ and, conditionally on the event that they are disjoint, by invariance under re-rooting the law of $ \Ball_{r}^{\bullet}( \Q_{\infty,p}^{{(k)}})$ is the same as that of $ \Ball_{r}^{\bullet}( \Q_{\infty,p+ \ell_{ \mathrm{in}}-  \ell_{ \mathrm{out}}})$, which converges to the law of $ \Ball_{r}^{\bullet}( \UIHPQ)$ by \eqref{def:UIPQ}. In particular this shows that $ \Ball_{r}^{\bullet}( \Q_{\infty,p}^{{(k)}})$ is asymptotically independent of $\{\Ball_r^\bullet( \Q_{\infty,p}) = \mathfrak{a}\}$ as $k \to \infty$ with $2p-k \to \infty$. This yields our claim. \endproof  

\subsection{Annealed and quenched connective constants}
\label{sec:connectiveconstant}

In a lattice, the connective constant is generally defined as the exponential growth rate (when it exists) of the number of $(0,n)$-SAWs. In our context, given an infinite quadrangulation $\mathfrak{q}$, we call the connective constant of $ \mathfrak{q}$ the quantity
$$\mu( \mathfrak{q}) = \limsup_{n \to \infty} \big(\# \SAW{0}{n}( \mathfrak{q}) \big)^{1/n}.$$
\begin{proposition}[Existence of the quenched connective constant] The connective constant $\mu(\Q_{\infty})$ of the UIPQ is almost surely constant.
\end{proposition}
\proof It follows from Lemma 2.1 of \cite{Lac12} that the value of the connective constant on an infinite connected locally finite graph does not depend on the starting point  of the self-avoiding walks (actually Lacoin assumes a uniform bound on the degrees but local finiteness is sufficient for the proof). In particular, the value of the connective constant on the UIPQ is invariant by changing the root edge. By ergodicity of the UIPQ (see \cite[Theorem 7.2]{Ang03} for the case of triangulations, which is easily adapted to our quadrangular case) any random variable which is invariant under changing the root edge must be almost surely constant (see \cite[Theorem 3.1]{AHNR15}). Hence $\mu(\Q_{\infty})$ is almost surely constant. \endproof

Although the value of the almost sure connective constant of the UIPQ (sometimes called the quenched connective constant) remains a mystery, we can precisely compute the average number of self-avoiding walks of any given type in the UIPQ.
\begin{proposition}[Annealed connective constant] With the same notation as \eqref{asympcp}, for any $b \geq 0$ and $f \geq 1$ we have 
$$ \mathbb{E}[\# \SAW{b}{f}(\Q_{\infty})] = \frac{C_{b+f}}{C_{1}}  = \left(\frac{9}{2}\right)^{b+f+o(b+f)}.$$
Hence we could say that the ``annealed'' connective constant of the UIPQ is $9/2$.
\end{proposition}
\proof  Let $b\geq 0$ and $f \geq 1$. Having fixed $n$, if $\Q_{n}$ is a uniform quadrangulation of the sphere with $n$ faces, by the bijection between quadrangulations endowed with a $(b,f)$-SAW and quadrangulations of the $2(b+f)$-gon we have (thanks to \eqref{eq:density1} and \eqref{asympcp})
$$ \mathbb{E}[\# \SAW{b}{f}(\Q_{n})] = \frac{\# \sQ_{n,b+f}}{ \#  \sQ_{n,1}}  \xrightarrow[n\to\infty]{} \frac{C_{b+f}}{C_{1}}.$$
On the other hand, the convergence of uniform quadrangulations towards the UIPQ implies that for any fixed $b,f$ the random variables $\# \SAW{b}{f}(\Q_{n})$ converge in law towards $\# \SAW{b}{f}(\Q_{\infty})$ as $n \to \infty$. The statement of the proposition thus follows once we prove that $(\# \SAW{b}{f}(\Q_{n}))_{n \geq 0}$ is uniformly integrable. In other words, for any $ \varepsilon >0$ we want to find $ A>0$ such that $ \mathbb{E}[\# \SAW{b}{f}(\Q_{n}) \mathbf{1}_{\# \SAW{b}{f}(\Q_{n}) > A}] \leq \varepsilon$ for all $n \geq 0$. If we denote $ \Z{b}{f}(\Q_{n,b+f}) =(\Q_{n}^{b,f}, \mathbf{w}^{b,f}_{n})$ we re-express the last quantity using the fact \eqref{eq:density1} that the density of $\Q_{n}^{b,f}$ with respect to $\Q_{n}$ is  proportional to $\# \SAW{b}{f}(\Q_{n})$:
\begin{eqnarray} \label{eq:uniforintegrable}
\mathbb{E}\big[\# \SAW{b}{f}(\Q_{n}) \1_{\# \SAW{b}{f}(\Q_{n}) > A}\big] & \underset{\eqref{eq:density1}}{=}  \frac{ \# \sQ_{n,1}}{ \# \sQ_{n,(b+f)}} \mathbb{E}\big[\mathbf{1}_{\# \SAW{b}{f}\big(\Q_{n}^{b,f}\big) > A}\big].  \end{eqnarray} 
Since we know that $\Q_{n}^{b,f}$ converges locally in distribution (see the discussion above Proposition \ref{prop:annealedSAW}), it follows that $\big(\# \SAW{b}{f}\big(\Q_{n}^{b,f}\big)\big)_{n \geq 1}$ converges in distribution as well and in particular is tight.  Using this fact and the asymptotics \eqref{asympqnpt} and \eqref{asympcp} we can find $A$ large enough so that the right-hand side of \eqref{eq:uniforintegrable} is less than $\varepsilon$ uniformly in $n \geq 0$ as desired. \endproof 

\begin{open}[Coincidence of the quenched and annealed connective constants] Combining the last two results we have $\mu(\Q_{\infty}) \leq 9/2$. Do we actually have a strict inequality?
\end{open}
\paragraph{Order and Disorder.}
The question of the coincidence of the quenched and annealed connective constants is usually referred to as weak/strong disorder in the statistical physics literature, see e.g.~\cite{Lac12,LDM91}. However, our context is different from the standard one where an underlying probability measure is tilted via a martingale biasing, so we shall use this section to clarify what we mean here by disorder. \medskip 

For simplicity we restrict ourselves to the case of a two-sided SAW in order to connect this section with Theorem \ref{thm:singular}. To simplify notation a little, we shall write $\Q_{\infty}^{p,\leftrightarrow}$ for the underlying rooted quadrangulation of $ \Z{p}{p}(\Q_{\infty,2p})$. Since the random variable giving the number of self-avoiding paths of a given type is continuous for the local topology, we can combine Proposition~\ref{prop:annealedSAW} with \eqref{eq:density1} to deduce that the Radon-Nikodym derivative of $\Q_{\infty}^{p,\leftrightarrow}$ with respect to $\Q_{\infty}$ is given by 
$$ \frac{ \mathrm{d} \Q_{\infty}^{p,\leftrightarrow}}{ \mathrm{d}\Q_{\infty}} = \frac{C_{1}}{C_{2p}} \# \SAW{p}{p}(\Q_{\infty}).$$
Hence $\Q_{\infty}^{p,\leftrightarrow}$ is only a ``mild'' modification of $\Q_{\infty}$, since the laws of the two random quadrangulations are equivalent. However, the distortion effect might become dramatic as $p \to \infty$. Borrowing terminology from statistical physics, we will say that \emph{disorder} holds if the law of $\Q_{\infty}^{\leftrightarrow} = \lim_{p} \Q_{\infty}^{p,\leftrightarrow}$ is singular with respect to that of $\Q_{\infty}$. This is exactly the content of Theorem \ref{thm:singular} which we will prove below. First, however, let us state a direct corollary:

\begin{cor} \label{cor:fewSAW} We have $\frac{C_{1}}{C_{2p}} \# \SAW{p}{p}(\Q_{\infty}) \to 0$ in probability as $p\to \infty$. \end{cor}
In other words, as $p \to \infty$ the typical number of $(p,p)$-SAWs on the UIPQ becomes much less than its expectation. Notice that even when disorder holds, the quenched and the annealed connective constants may very well be equal.

\proof We prove the result in greater generality. Let $(E,d)$ be a Polish space and $\mu,\nu,(\mu_{n})_{n\geq 0}$ be probability measures on $E$ such that $\mu_{n} \to \mu$ in distribution as $n\to\infty$ and such that $\mu_{n}$ is absolutely continuous with respect to $\nu$, with density $f_{n}$ (here $\mu_{p}$ is the law of $\Q_{\infty}^{p,\leftrightarrow}$, $\mu$ the law of $\Q_{\infty}^{\leftrightarrow}$ and $\nu$ that of the UIPQ). Assuming $\nu$ and $\mu$ are singular with respect to each other, the goal is to prove that 
$$ f_{n}  \to 0, \quad \mbox{ in probability for } \nu.$$
Notice that, in general, there is no equivalence between the fact that $\nu$ and $\mu$ are singular and the fact that $f_{n} \to 0$ in $\nu$-probability. We pick a measurable subset $A$ such that $\mu(A)=0$ and $\nu(A)=1$. By regularity we can find a closed subset $F \subseteq A$ such that $\nu(F) \geq 1- \varepsilon$ and a fortiori $\mu(F) = 0$. By the Portmanteau theorem we thus have 
$$ 0=\mu(F) \geq \limsup_{n \to \infty} \mu_{n}(F) = \limsup_{n \to \infty} \int \mathrm{d}\nu \, f_{n} \mathbf{1}_{F}.$$
It follows that $ \nu( \{f_{n} \geq \varepsilon\}) \leq \nu( \{f_{n} \geq \varepsilon\} \cap F) + \varepsilon \leq  \varepsilon^{-1} \int \mathrm{d}\nu \, f_{n} \mathbf{1}_{F} + \varepsilon$, which is eventually less than $2 \varepsilon$ by the above display. We have thus proved that $f_{n} \to 0$ in $\nu$-probability as desired. \endproof

\section{Displacement of the distinguished SAW in $\Q_\infty^\rightarrow$ and $\Q_\infty^\leftrightarrow$} \label{sec:displacement}
This section is devoted to proving Theorem \ref{thm:diffusive}. Recall the notation $ (\mathsf{P}_{i}^{\rightarrow})_{i \geq 0}$ and $ (\mathsf{P}^{\leftrightarrow}_{i})_{i \in \mathbb{Z}}$ for the distinguished self-avoiding walks on $ \Q_{\infty}^{\rightarrow}$ and $\Q_{\infty}^{\leftrightarrow}$ respectively. From our previous paper \cite{CCuihpq}, the following is easily inferred:
\begin{eqnarray} \label{cor:old bound} \mathrm{d}_{ \mathrm{gr}}^{\Q_{\infty}^{\rightarrow}}( \mathsf{P}^\rightarrow_{0}, \mathsf{P}^\rightarrow_{n})  \preceq  \sqrt{n} & \mbox{ and }& \mathrm{d}_{ \mathrm{gr}}^{\Q_{\infty}^{\leftrightarrow}}( \mathsf{P}^\leftrightarrow_{0}, \mathsf{P}^\leftrightarrow_{\pm n})  \preceq    \sqrt{n}, \qquad n\geq 0,  \end{eqnarray} 
where the notation $ \mathrm{d}_{ \mathrm{gr}}^{ \mathfrak{q}}(\vec{u},\vec{v})$ stands for the minimum graph distance between an endpoint of $\vec{u}$ and an endpoint of $\vec{v}$ in the quadrangulation $ \mathfrak{q}$. This is quite immediate from \cite[Proposition 6.1]{CCuihpq}: if we write $(x_{i})_{ i \in \mathbb{Z}}$ for the boundary vertices of a UIHPQ $\UIHPQ$ (labelling them in the natural way, so that $x_{0} \to x_{1}$ is the root edge) then the construction of $\Q_{\infty}^{\rightarrow}$ from $\UIHPQ$ may only decrease distances, so that
$$ \dgr^{\Q_{\infty}^{\rightarrow}}(\mathsf{P}_{0}^{\rightarrow}, \mathsf{P}_{n}^{\rightarrow}) \leq \dgr^{\UIHPQ}(x_{0},x_{n}).$$
Thus $\mathrm{d}_{ \mathrm{gr}}^{\Q_{\infty}^{\rightarrow}}(\mathsf{P}_{0}^{\rightarrow}, \mathsf{P}_{n}^{\rightarrow}) \preceq \sqrt{n}$, and the case of $\Q_\infty^\leftrightarrow$ is analogous.\medskip

To establish Theorem \ref{thm:diffusive} what we wish to obtain is a corresponding lower bound
\begin{eqnarray} \label{prop:lower bound} \mathrm{d}_{ \mathrm{gr}}^{\Q_{\infty}^{\rightarrow}}( \mathsf{P}^\rightarrow_{0}, \mathsf{P}^\rightarrow_{n})  \succeq  \sqrt{n} & \mbox{ and }& \mathrm{d}_{ \mathrm{gr}}^{\Q_{\infty}^{\leftrightarrow}}( \mathsf{P}^\leftrightarrow_{0}, \mathsf{P}^\leftrightarrow_{\pm n})  \succeq    \sqrt{n}, \qquad n\geq 0.  \end{eqnarray} 
We shall establish such a bound by constructing a sequence of nested ``fences'' $(P'_i)_{i \geq 1}$ inside $\UIHPQ$, that is a sequence of disjoint paths whose endpoints lie on the boundary of $\UIHPQ$ and are of the form $x_{-r(i)}$ and $x_{r(i)}$ for an increasing integer sequence $r(i)$. After the folding of $\UIHPQ$ to form $\Q_{\infty}^{\rightarrow}$ these paths will create nested loops so that $k \geq r(i)$ implies $ \dgr(x_{0},x_{k}) \geq i$, see Figure \ref{fig:fences one-ended}.

The construction of these fences will be achieved by a (deterministic) algorithm which, when applied to the UIHPQ, will yield a random sequence $(P'_i)_{i \geq 1}$ for which one wishes to control the ``growth'' on the boundary $(r(i))_{ i \geq 1}$. This will be possible thanks to the spatial Markov property of the simple boundary UIHPQ.

\subsection{Building one fence}\label{algorithm}

Suppose you are given a (deterministic) one-ended quadrangulation $ \mathfrak{h}$ with an infinite simple boundary on the sequence of (successive) vertices $(v_i)_{i\in\mathbb{Z}}$ so that $v_{0} \to v_{1}$ is the root edge, and a positive integer $k$.

We shall build a `fence' $P$ that avoids vertices $v_1,\ldots,v_k$ via a peeling process. We first set $e_0$ to be the root edge and then iteratively perform the following loop, starting with $i=0$:
\begin{itemize}
	\item reveal the face $f_i$ lying left of $e_i$; 
	\item consider the \emph{rightmost} vertex $v_{\rho_i}$ of $f_i$ lying on the boundary of $ \mathfrak{h}$; set $e_{i+1}$ to be the \emph{rightmost} edge in $ \mathfrak{h}$ which has $v_{\rho_{i}}$ as an endpoint and belongs to $f_i$ (oriented towards $v_{\rho_{i}}$);
	\item if $\rho_{i}>k$, then \texttt{STOP}; otherwise restart the loop, increasing $i$ by 1.
\end{itemize}

Notice that each revealed face does have a vertex on the boundary, thus the operations required are well defined, and that the sequence $(\rho_{i})_{i\geq 0}$ is weakly increasing, so that (thanks to local finiteness of $ \mathfrak{h}$) the algorithm does eventually terminate (see Figure \ref{fig:symmetry}).

Once the end condition is met (at -- say -- iteration $T$, where iterations are numbered from $0$), we have a final (connected) set of revealed faces $F=\{f_0,\ldots,f_T\}$. We consider then the hull $F^{\bullet}$ of $F$ obtained by ``filling in'' any finite holes between $F$ and the boundary of $\mathfrak{h}$, and set our fence $P$ to be the inner boundary of $ F^{\bullet}$ (i.e.~the part of the boundary of $F^{\bullet}$ which is not in common with the boundary of $ \mathfrak{h}$). It is easy to show that $P$ is indeed a simple path, that one of its endpoints is $v_{r+k}$ for some $r \geq 1$ while the other is some $v_{1-\ell}$ with $\ell\geq 1$, and that it has no vertices on the boundary of $ \mathfrak{h}$ except for its endpoints, so that it does not intersect $\{v_1,\ldots,v_k\}$.

\begin{definition} We call the quantities $\ell$ and $r$ respectively the \emph{left} and \emph{right} \emph{overshoots} of the construction. \end{definition}

\begin{rem}\label{symmetry of fence-building}We make here an alternative direct definition of $F^{\bullet}$ which will be useful later. We claim that $ F^{\bullet}$ is also the hull of the set of faces of the quadrangulation having (at least) a vertex in the set $\{v_1,\ldots,v_k\}$. One inclusion is clear, since all faces with a vertex in $\{v_1,\ldots,v_k\}$ must lie below $P$, which separates $\{v_1,\ldots,v_k\}$ from infinity; the other is also clear, since all faces revealed during the construction of $P$ have a vertex in the set $\{v_{1}, ... , v_{k}\}$.

This alternative construction of $P$ highlights the inherent symmetry in the roles of the left and right overshoots. If we flip the quadrangulation $ \mathfrak{h}$ (exchanging left and right) and relabel its boundary vertices as $(v'_i)_{i\in\mathbb{Z}}$ so that $v'_i$ is $v_{k+1-i}$, then perform the algorithm to build a fence as above (starting with $e'_0=(v'_0,v'_1)$, which corresponds to $(v_{k+1},v_k)$), then the left overshoot of this new fence is the right overshoot of $P$ and vice versa (see Figure~\ref{fig:symmetry}).
\end{rem}

\begin{figure}\centering
	\begin{tikzpicture}
	\begin{pgfonlayer}{nodelayer}
		\node [style=real, label=below:$v_5$] (0) at (5, -0) {};
		\node [style=real, label=below:$v_6$] (1) at (6, -0) {};
		\node [style=real, label=below:$v_4$] (2) at (4, -0) {};
		\node [style=real, label=below:$v_3$, fill=red] (3) at (3, -0) {};
		\node [style=real, label=below:$v_2$, fill=red] (4) at (2, -0) {};
		\node [style=real, label=below:$v_1$, fill=red] (5) at (1, -0) {};
		\node [style=real, label=below:$v_0$] (6) at (0, -0) {};
		\node [style=real, label=below:$v_{-1}$] (7) at (-1, -0) {};
		\node [style=real, label=below:$v_{-2}$] (8) at (-2, -0) {};
		\node [style=real, label=below:$v_{-3}$] (9) at (-3, -0) {};
		\node [style=real, label=below:$v_{-4}$] (10) at (-4, -0) {};
		\node [style=real] (11) at (3.25, 1.75) {};
		\node [style=real] (12) at (3.25, 1) {};
		\node [style=real] (13) at (4.5, 0.25) {};
		\node [style=real] (14) at (-1.25, 2.25) {};
		\node [style=real] (15) at (1, 3) {};
		\node [style=real] (16) at (0.5, 0.25) {};
		\contourlength{2pt}
		\node [style=none] (17) at (0.5, -0.3) {\contour{white}{$e_0$}};
		\node [style=none] (18) at (0.75, 0.5) {\contour{white}{$e_1$}};
		\node [style=none] (19) at (0, 0.75) {\contour{white}{$e_2$}};
		\node [style=none] (20) at (0.5, 1.75) {\contour{white}{$e_3$}};
		\node [style=none] (21) at (2.5, 1.5) {\contour{white}{$e_4$}};
	\end{pgfonlayer}
	\begin{pgfonlayer}{edgelayer}
	\fill[gray!30] (12.center) to (4.center) to (7.center) [bend right=60, looseness=1.25] to (10.center) [bend left=30, looseness=1.00] to (14.center)--(15.center)--(11.center) [bend left, looseness=1.00] to (1.center) [bend right, looseness=0.75] to (12.center);

		\draw [b, bend left, looseness=1.00] (11) to (1);
		\draw [b, bend left=45, looseness=1.00] (4) to (11);
		\draw[b] (4) to (3);
		\draw[b] (3) to (2);
		\draw [b, bend left=75, looseness=1.75] (2) to (0);
		\draw[b] (0) to (1);
		\draw[b] (12) to (4);
		\draw [b, bend right, looseness=0.75] (1) to (12);
		\draw[b] (12) to (2);
		\draw[b] (13) to (0);
		\draw[b] (0) to (2);
		\draw [b, bend right=45, looseness=1.25] (4) to (7);
		\draw [b, bend right=60, looseness=1.25] (7) to (10);
		\draw [b, bend right, looseness=1.00] (14) to (10);
		\draw [b, bend left, looseness=1.00] (14) to (4);
		\draw[b] (14) to (15);
		\draw[b] (15) to (11);
		\draw[b] (10) to (9);
		\draw[b] (9) to (8);
		\draw[b] (8) to (7);
		\draw [b, bend left=90, looseness=1.75] (6) to (5);
		\draw[b] (7) to (6);
		\draw[b] (5) to (4);
		\draw[b] (16) to (6);
		\draw [b] (6) to (5);
		
		\draw [bend left, looseness=1.00] (11) to (1);
		\draw [<-, ultra thick, bend left=45, looseness=1.00] (4) to (11);
		\draw [red](4) to (3);
		\draw (3) to (2);
		\draw [bend left=75, looseness=1.75] (2) to (0);
		\draw (0) to (1);
		\draw (12) to (4);
		\draw [bend right, looseness=0.75] (1) to (12);
		\draw (12) to (2);
		\draw (13) to (0);
		\draw (0) to (2);
		\draw [<-, ultra thick, bend right=45, looseness=1.25] (4) to (7);
		\draw [bend right=60, looseness=1.25] (7) to (10);
		\draw [bend right, looseness=1.00] (14) to (10);
		\draw [->, ultra thick, bend left, looseness=1.00] (14) to (4);
		\draw (14) to (15);
		\draw (15) to (11);
		\draw (10) to (9);
		\draw (9) to (8);
		\draw (8) to (7);
		\draw [->, ultra thick, bend left=90, looseness=1.75] (6) to (5);
		\draw (7) to (6);
		\draw [red](5) to (4);
		\draw (16) to (6);
		\draw[->, ultra thick] (6) to (5);
	\end{pgfonlayer}
\end{tikzpicture}\\
	\begin{tikzpicture}[all edges/.style={very thick}]
	\begin{pgfonlayer}{nodelayer}
		\node [style=real, label=below:$v'_{-2}$] (0) at (5, -0) {};
		\node [style=real, label=below:$v'_{-1}$] (1) at (6, -0) {};
		\node [style=real, label=below:$v'_0$] (2) at (4, -0) {};
		\node [style=real, label=below:$v'_1$, fill=red] (3) at (3, -0) {};
		\node [style=real, label=below:$v'_2$, fill=red] (4) at (2, -0) {};
		\node [style=real, label=below:$v'_3$, fill=red] (5) at (1, -0) {};
		\node [style=real, label=below:$v'_4$] (6) at (0, -0) {};
		\node [style=real, label=below:$v'_{5}$] (7) at (-1, -0) {};
		\node [style=real, label=below:$v'_{6}$] (8) at (-2, -0) {};
		\node [style=real, label=below:$v'_{7}$] (9) at (-3, -0) {};
		\node [style=real, label=below:$v'_{8}$] (10) at (-4, -0) {};
		\node [style=real] (11) at (3.25, 1.75) {};
		\node [style=real] (12) at (3.25, 1) {};
		\node [style=real] (13) at (4.5, 0.25) {};
		\node [style=real] (14) at (-1.25, 2.25) {};
		\node [style=real] (15) at (1, 3) {};
		\node [style=real] (16) at (0.5, 0.25) {};
		\contourlength{2pt}
		\node [style=none] (17) at (3.5, -0.3) {\contour{white}{$e_0$}};
		\node [style=none] (18) at (2.75, 0.5) {\contour{white}{$e_1$}};
		\node [style=none] (19) at (2.25, 1.25) {\contour{white}{$e_2$}};
		\node [style=none] (20) at (0.5, 1.75) {\contour{white}{$e_3$}};
	\end{pgfonlayer}
	\begin{pgfonlayer}{edgelayer}
	\fill[gray!30] (12.center) to (2.center) to (4.center) [bend right=45, looseness=1.25] to (7.center) [bend right=60, looseness=1.25] to (10.center) [bend left=30, looseness=1.00] to (14.center)--(15.center)--(11.center) [bend left, looseness=1.00] to (1.center) [bend right, looseness=0.75] to (12.center);

		\draw [b, bend left, looseness=1.00] (11) to (1);
		\draw [b, bend left=45, looseness=1.00] (4) to (11);
		\draw[b] (4) to (3);
		\draw[b] (3) to (2);
		\draw [b, bend left=75, looseness=1.75] (2) to (0);
		\draw[b] (0) to (1);
		\draw[b] (12) to (4);
		\draw [b, bend right, looseness=0.75] (1) to (12);
		\draw[b] (12) to (2);
		\draw[b] (13) to (0);
		\draw[b] (0) to (2);
		\draw [b, bend right=45, looseness=1.25] (4) to (7);
		\draw [b, bend right=60, looseness=1.25] (7) to (10);
		\draw [b, bend right, looseness=1.00] (14) to (10);
		\draw [b, bend left, looseness=1.00] (14) to (4);
		\draw[b] (14) to (15);
		\draw[b] (15) to (11);
		\draw[b] (10) to (9);
		\draw[b] (9) to (8);
		\draw[b] (8) to (7);
		\draw [b, bend left=90, looseness=1.75] (6) to (5);
		\draw[b] (7) to (6);
		\draw[b] (5) to (4);
		\draw[b] (16) to (6);
		\draw [b] (6) to (5);
		
		\draw [bend left, looseness=1.00] (11) to (1);
		\draw [<-, ultra thick, bend left=45, looseness=1.00] (4) to (11);
		\draw [red](4) to (3);
		\draw [<-, ultra thick] (3) to (2);
		\draw [bend left=75, looseness=1.75] (2) to (0);
		\draw (0) to (1);
		\draw [->, ultra thick] (12) to (4);
		\draw [bend right, looseness=0.75] (1) to (12);
		\draw (12) to (2);
		\draw (13) to (0);
		\draw (0) to (2);
		\draw [bend right=45, looseness=1.25] (4) to (7);
		\draw [bend right=60, looseness=1.25] (7) to (10);
		\draw [bend right, looseness=1.00] (14) to (10);
		\draw [->, ultra thick, bend left, looseness=1.00] (14) to (4);
		\draw (14) to (15);
		\draw (15) to (11);
		\draw (10) to (9);
		\draw (9) to (8);
		\draw (8) to (7);
		\draw [bend left=90, looseness=1.75] (6) to (5);
		\draw (7) to (6);
		\draw [red](5) to (4);
		\draw (16) to (6);
		\draw (6) to (5);
	\end{pgfonlayer}
\end{tikzpicture}
\caption{\label{fig:symmetry}Above, the algorithm run with $k=3$ reveals 5 faces; the left overshoot is $|-4|+1=5$ and the right overshoot is $6-3=3$. Below, the flipped algorithm reveals only 3 faces, but produces the same ``fence'', with a left overshoot of $|-2|+1=3$ and a right overshoot of $8-3=5$.}
\end{figure}
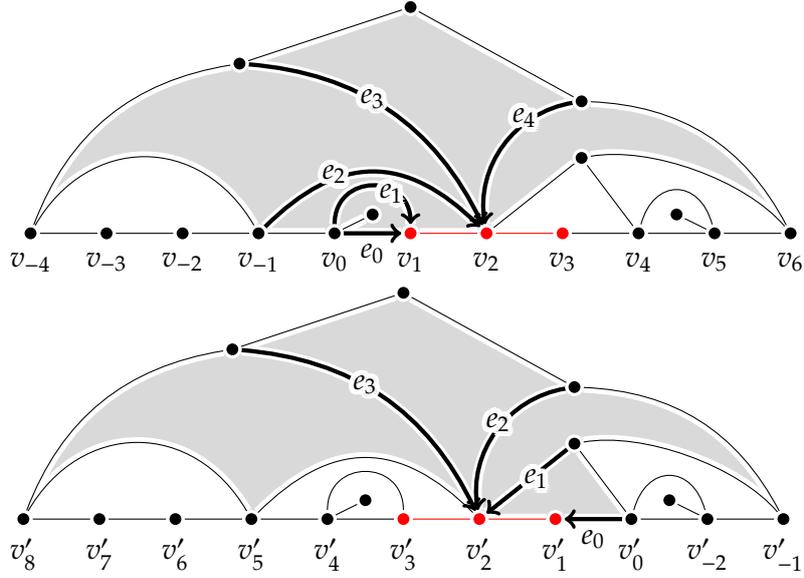

\subsection{One fence in the simple boundary UIHPQ}

Given $k \geq 0$ and a copy of the simple boundary UIHPQ $\UIHPQ$, whose boundary vertices we call $(x_i)_{i\in\mathbb{Z}}$, we can use the above algorithm to discover $F^{\bullet}$ and build a fence $P$ enclosing the vertices $x_{1}, \ldots, x_{k}$. The spatial Markov property of the UIHPQ already used in the proof of Lemma \ref{lem:indept} then show that conditionally on $F^{\bullet}$ the remaining part $ \UIHPQ \backslash F^{\bullet}$ (rooted for example at the rightmost edge on the boundary of $\UIHPQ$ which lies to the left of $F^{\bullet}$) has the law of a UIHPQ. The easiest way to see this is to say that $F^{\bullet}$ has been discovered by the mean of a peeling process (we shall discuss this later in more detail).

In this setting we denote by $X^{+,(k)}$ and $X^{-,(k)}$ respectively the right and left overshoots in the construction of $P$. Obviously the law of these overshoots depends on $k$, but the first observation we can make is that Remark~\ref{symmetry of fence-building} (combined with the fact that the law of the UIHPQ is invariant under ``flipping'')  implies that for all $k \geq 1$ 
 \begin{eqnarray} \label{eq:symm} X^{+,(k)} = X^{-,(k)} \qquad \mbox{ in distribution}.  \end{eqnarray}

This being said, we will stochastically bound the variable $X^{-,(k)}$ by a random variable which is independent of $k$. To do so, we may consider the peeling process on the UIHPQ which consists in running the loop of Section~\ref{algorithm} indefinitely (roughly speaking by setting $k=\infty$), and set $X$ to be the (random) index of the leftmost boundary vertex of a face that is eventually revealed by it; it can be shown that $X$ is almost surely well defined (i.e.~finite), and naturally we have $X^{-,(k)}\leq |X|+1$ stochastically for all $k \geq 1$.

The explicit construction of the peeling process makes it possible to show the following:
\begin{lemma}\label{lem:overshoot estimate}
For the peeling process based on Section~\ref{algorithm} applied to $\UIHPQ$ with $k = \infty$, the (random) index $X$ of the leftmost boundary vertex belonging to a  face that is eventually revealed is such that
$$\sup_{n\geq 0}\sqrt{n}\P(X\leq -n)<\infty.$$
\end{lemma}

For the sake of completeness, we shall first devote a subsection to an explicit description of the peeling process (based on \cite{ACpercopeel}, to which we refer the reader for details), and then use it to give a proof of the above lemma. Before doing so, we deduce the technical corollary that we will use:
\begin{cor} \label{cor:coupling} There exists a (law of a) random variable $O \geq 1$ whose tail satisfies $ \mathbb{P}( O \geq n) \leq C n^{-1/2}$ for some $C>0$ and such that for any $k \geq 1$ we can couple $O$ and the overshoots $(X^{-,(k)},X^{+,(k)})$ so that  $$ \max(X^{-,(k)},X^{+,(k)})  \leq O.$$
\end{cor}
\proof It suffices to estimate the tail of $\max(X^{-,(k)},X^{+,(k)})$. For $n \geq 1$ we have 
  \begin{eqnarray*} \mathbb{P}(\max(X^{-,(k)},X^{+,(k)}) \geq n) &\leq&  \mathbb{P}( X^{+,(k)} \geq n/2)+  \mathbb{P}( X^{-,(k)} \geq n/2)\\
  & \underset{ \eqref{eq:symm}}{=}& 2\mathbb{P}( X^{-,(k)} \geq n/2)\\
  & \leq & 2\mathbb{P}( -X \geq n/2-1) \underset{ \mathrm{Lem.\ } \ref{lem:overshoot estimate} }{\leq} C n^{-1/2}.  \end{eqnarray*}
  The statement of the corollary then follows from standard coupling arguments.  \endproof

\subsection{The peeling process}\label{sec:peeling process}
Let $\UIHPQ$ be a simple boundary UIHPQ and let us reveal the quadrangular face that contains the root edge (an operation which we may call \emph{peeling} the root edge). The revealed face can separate the map into one, two or three regions, only one of which is infinite according to the cases listed below. Conditionally on each of these cases, such regions are independent from each other, the infinite one (in light grey in the following figures) always being a copy of a UIHPQ while the finite ones (in dark grey) are Boltzmann quadrangulations of appropriate perimeter. We shall call edges of the revealed face that belong to the infinite region \emph{exposed} edges; edges of the boundary of $\UIHPQ$ that belong to the finite regions will be said to have been \emph{swallowed}, and we will distinguish edges that are swallowed \emph{to the left} and \emph{to the right} according to whether they lie to the left or right of the root edge being peeled (see Figure~\ref{exposed-swallowed}). We distinguish the following cases:
\begin{itemize}
\item[$\mathsf{C}$] Firstly, the revealed face may have exactly two vertices on the boundary of $\UIHPQ$. In this case we say that the form of the quadrangle revealed is $\mathsf{C}$ (for center). 
  \begin{center}
    \includegraphics[width=45mm]{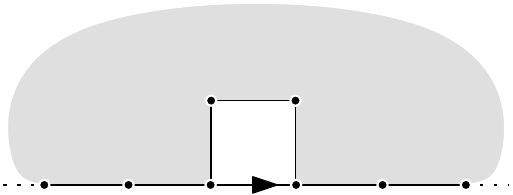}
  \end{center}
  
\item[$\mathsf{R}_\ell, \mathsf{L}_\ell$] The revealed face can also have three of its vertices lying on the
  boundary of $\UIHPQ$ and one in the interior, thus separating the map into a region with a finite boundary and one with an infinite boundary.  We have two sub-cases, depending on whether the third vertex lies on the left (case $ \mathsf{L}_{\ell}$) or on the right (case $ \mathsf{R}_{\ell}$) of the root edge. Suppose for example that the third vertex lies
  on the left of the root edge; the fourth vertex of the quadrangle may lie on the boundary of the finite region or of the infinite region.  Since all quadrangulations are
  bipartite, this is determined by the parity of the number $\ell$ of \emph{swallowed} edges (see the figure below for the case of $ \mathsf{L}_{\ell}$).

  \begin{center}
    \includegraphics[width=11cm]{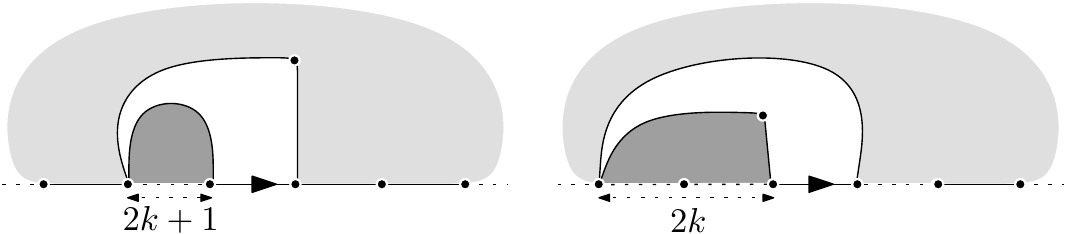}
  \end{center}

\item[$\mathsf{L}_{\ell_{1},\ell_{2}}$, $\mathsf{C}_{\ell_1,\ell_2}$, $\mathsf{R}_{\ell_{1},\ell_{2}}$] The last case to consider is when the revealed quadrangle has all of its
  four vertices on the boundary.  In this case the revealed face
  separates from infinity two segments of length $\ell_1$
  and $\ell_2$ along the boundary, as depicted in the figure below. This could happen in three ways, as $0$,
  $1$, or $2$ vertices could lie to the right of the root edge (see
  \cref{fig:peelhalfquad3}) and the corresponding subcases are denoted by $ \mathsf{L}_{\ell_{1},\ell_{2}}$, $  \mathsf{C}_{\ell_{1},\ell_{2}}$ and $ \mathsf{R}_{\ell_{1},\ell_{2}}$.  Notice that the numbers $\ell_1=2k_{1}+1$ and $\ell_2=2k_{2}+1$
  must both be odd. 

  \begin{figure}[h]
    \begin{center}
      \includegraphics[width=12cm]{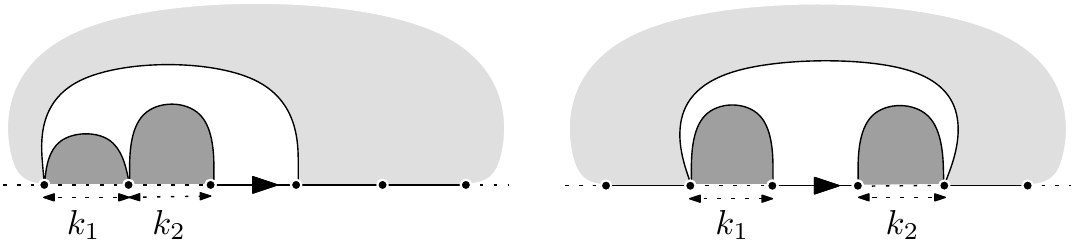}
      \caption{Cases $ \mathsf{L}_{\ell_{1},\ell_{2}}$ and $\mathsf{C}_{\ell_{1},\ell_{2}}$.
      \label{fig:peelhalfquad3}}
    \end{center}
  \end{figure}
\end{itemize}

\begin{figure}[!h]
  \begin{center}
    \includegraphics[width=\textwidth]{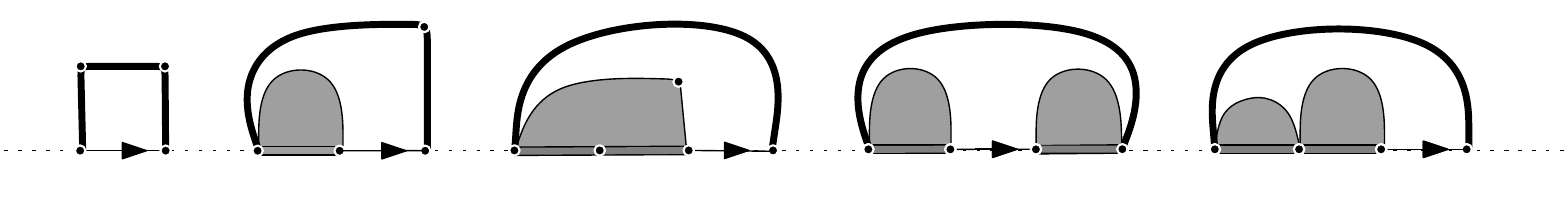}
    \caption{The exposed edges are in fat black lines and the swallowed
      ones are in fat gray lines (the remaining cases are symmetric).}
    \label{exposed-swallowed}
  \end{center}
\end{figure}
The exact probabilities of these events can be computed explicitly (see \cite{ACpercopeel}) but we will only use the fact \cite[Section 2.3.2]{ACpercopeel} that, if $ \mathcal{E}$ and $ \mathcal{S}$ denote the number of edges respectively exposed and swallowed by a peeling step, 
 \begin{eqnarray} \label{eq:infopeeling} \mathbb{E}[ \mathcal{E}] = 2, \quad \mathbb{E}[  \mathcal{S}] = 1, \quad \mbox{ and }\quad  	\mathbb{P}(  \mathcal{S} = k) \sim C k^{-5/2}  \end{eqnarray} for some constant $C>0$ as $k \to \infty$.

By iterating the one-step peeling described above one can define a growth algorithm that discovers (a subset of) the simple boundary UIHPQ step by step. 

A \emph{peeling process} is a randomized algorithm that  explores $\UIHPQ$ by revealing at each step the face in the unexplored part adjacent to a given edge, together with any finite regions that it encloses; in order to choose the next edge to peel, one can use the submap of $\UIHPQ$ that has already been revealed and possibly another source of randomness as long as the choice remains independent of the unknown region (see \cite[Section 2.3.3]{ACpercopeel} for details). Under these assumptions the one-step peeling transitions and the invariance of $\UIHPQ$ under re-rooting along the boundary show that the peeling steps are i.i.d.,~see \cite[Proposition 4]{ACpercopeel}.

 Notice that this is definitely the case with the algorithm described in Section~\ref{algorithm}, which in fact consists of the one-step peeling described above, with the chosen edge at each step being simply the rightmost exposed edge of the most recently revealed face (until the algorithm stops).
 
 Furthermore, consider the number of \emph{exposed}, \emph{left swallowed} and \emph{right swallowed} edges at each peeling step; each of these quantities is a random variable whose law can be computed explicitly thanks to the probabilities of the various events listed above, by referring to Table~\ref{table}.
 
\begin{table}[h]\centering
 	$\begin{array}{|l|c|c|c|c|}
     \hline
     \mbox{Case} & \mbox{Exposed} & \mbox{left swallowed} & \mbox{right swallowed} & \Delta Y\\
     \hline
      \mathsf{C} & 3 & 0 & 0 & 2\\
      \mathsf{L}_{2k} &  1 & 2k & 0 &  -2k\\
      \mathsf{L}_{2k+1} & 2 & 2k+1 & 0 &-2k\\
      \mathsf{R}_{2k} &  1 & 0 & 2k &0\\
      \mathsf{R}_{2k+1} &  2 & 0 & 2k+1 &1\\
       \mathsf{L}_{k_{1},k_{2}}&  1 &  k_{1}+k_{2} & 0 & -k_{1}-k_{2}\\
       \mathsf{C}_{k_{1},k_{2}} & 1 &  k_{1} & k_{2} & -k_{1}\\
\mathsf{R}_{k_{1},k_{2}} & 1 & 0 & k_{1}+ k_{2} &  0\\\hline\end{array}$
\caption{\label{table}}
\end{table}
 
\subsection{Overshoot estimates}

We now can proceed with the proof of Lemma~\ref{lem:overshoot estimate}.

\begin{proof}[Proof of Lemma~\ref{lem:overshoot estimate}]
Consider the peeling process on $ \UIHPQ$ as defined in Section~\ref{algorithm}, run indefinitely (with $k= \infty$). Steps are numbered from $0$, and step $i$ reveals a face $f_i$ and outputs an oriented edge $e_{i+1}$ incident to $f_i$, with an endpoint $x_{\rho_{i}}$ on the boundary $(x_i)_{i\in\mathbb{Z}}$ of $\UIHPQ$; the edge $e_{i+1}$ is then peeled at step $i+1$; we denote by $\UIHPQ(i)$ the map obtained from $\UIHPQ$ after removing the hull of the faces discovered up to time $i$. We already know that $\UIHPQ(i)$ -- appropriately rooted -- is distributed as a UIHPQ. We consider the section $\gamma_{i}$ of the boundary of $\UIHPQ(i)$ lying  left of $e_{i+1}$. This is of course made of infinitely many edges, but $\gamma_{i}$ and $\gamma_{0}$ differ by only  finitely many edges. This makes possible to define for each $i\geq 0$ a quantity $Y_{i}$ which represents the algebraic variation of the ``length'' of $\gamma_{i}$ with respect to $\gamma_{0}$. The formal definition of $Y_{i}$ is given from $Y_{0}=0$ via  its variation $\Delta Y_{i} = Y_{i+1}-Y_{i}$ equal to the number of exposed edges minus the number of swallowed edges on the left  of the current point minus $1$ (see the above table).
Clearly the definition of $(Y)$ and the properties of the peeling process entail that $(Y)$ is a random walk with i.i.d.~increments whose law can be explicitly computed. In particular, $ \mathbb{E}[\Delta Y] = \mathbb{E}[ \mathcal{E}] - 1 -  \mathbb{E}[ \mathcal{S}]/2 = \frac{1}{2}$ with the notation of \eqref{eq:infopeeling}. Since $(Y)$ has a positive drift we can define the overall infimum
$$ \inf_{i \geq 0} Y_{i} > -\infty.$$
An easy geometric argument then shows that the variable $X$ we are after is just $\inf_{i \geq 0} Y_{i}$. The tail of the overall infimum of the transient random walk $(Y)$ can be estimated from the tail of $ \Delta Y$ (which is given by \eqref{eq:infopeeling}) using \cite[Theorem 2]{Ver77}. It follows that for some $C'>0$ $$\P(X\leq -n)=\P\left(\inf_{i\geq 0} Y_i\leq -n\right)\leq C'n^{-1/2}, \quad \mbox{ for all }n \geq 1.$$ \end{proof}

\subsection{Building the final fences in $\Q_\infty^\rightarrow$ and $\Q_\infty^{\leftrightarrow}$}

Now, in order to conclude our proof of \cref{thm:diffusive}, we need a little tweaking of the fence-building algorithm, so as to have fence endpoints coincide in the case of $\Q_\infty^\rightarrow$ and $\Q_\infty^{\leftrightarrow}$.

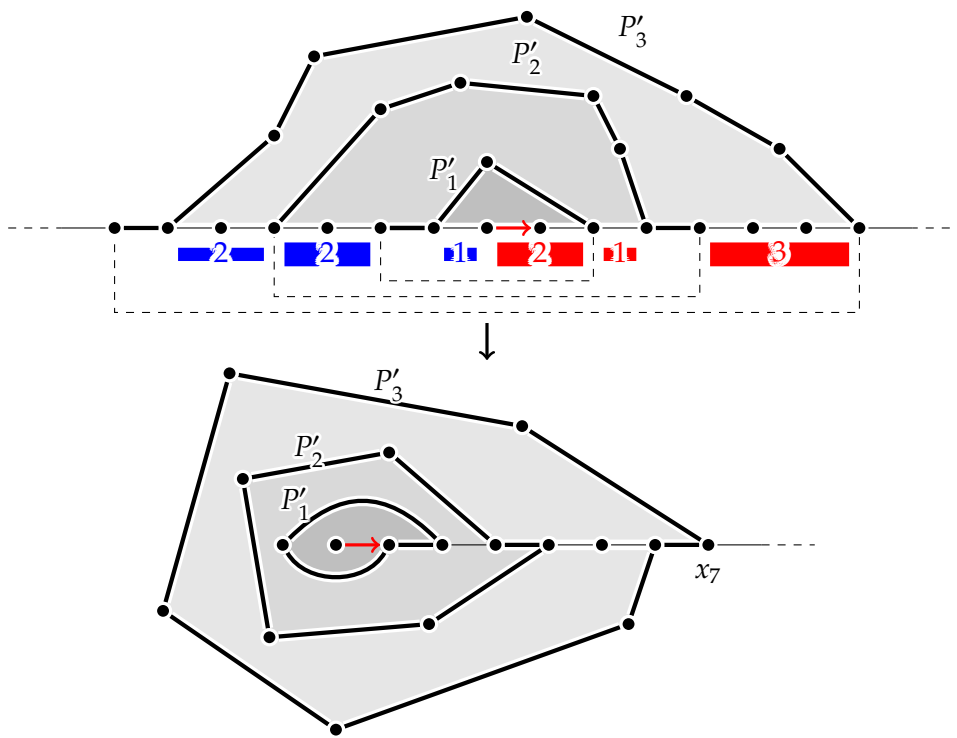
\begin{figure}[h!!]
	\centering
	\begin{tikzpicture}[cont/.style={ultra thick}, scale=.7]
	\begin{pgfonlayer}{nodelayer}
		\node [style=real] (0) at (0, -0) {};
		\node [style=real] (1) at (1, -0) {};
		\node [style=real] (2) at (2, -0) {};
		\node [style=real] (3) at (3, -0) {};
		\node [style=real] (4) at (4, -0) {};
		\node [style=real] (5) at (5, -0) {};
		\node [style=real] (6) at (6, -0) {};
		\node [style=real] (7) at (-1, -0) {};
		\node [style=real] (8) at (-2, -0) {};
		\node [style=real] (9) at (-3, -0) {};
		\node [style=real] (10) at (-4, -0) {};
		\node [style=real] (11) at (-5, -0) {};
		\node [style=real] (12) at (-6, -0) {};
		\node [style=real] (13) at (-2, 2.25) {};
		\node [style=real] (14) at (2, 2.5) {};
		\node [style=real] (15) at (-0.5, 2.75) {};
		\node [style=real] (16) at (2.5, 1.5) {};
		\node [style=real] (17) at (0, 1.25) {};
		\node [style=real] (18) at (-4, 1.75) {};
		\node [style=real] (19) at (-3.25, 3.25) {};
		\node [style=real] (20) at (0.75, 4) {};
		\node [style=real] (21) at (3.75, 2.5) {};
		\node [style=real] (22) at (5.5, 1.5) {};
		\node [style=real] (23) at (7, -0) {};
		\node [style=real] (24) at (-7, -0) {};
		\node (25) at (-1, -0.5) {};
		\node (26) at (0, -0.5) {};
		\node (27) at (2, -0.5) {};
		\node (28) at (-4, -0.5) {};
		\node (29) at (-2, -0.5) {};
		\node (30) at (3, -0.5) {};
		\node (31) at (-6, -0.5) {};
		\node (32) at (4, -0.5) {};
		\node (33) at (7, -0.5) {};
		\node (34) at (7, -1) {};
		\node (35) at (4, -1) {};
		\node (36) at (2, -1) {};
		\node (37) at (0, -1) {};
		\node (38) at (-2, -1) {};
		\node [] (39) at (-4, -1) {};
		\node [] (40) at (-7, -1) {};
		\contourlength{1pt}
				\node (41) at (-0.8, 1) {\contour{white}{$P'_1$}};
		\node (42) at (0.75, 3.2) {\contour{white}{$P'_2$}};
		\node (43) at (2.75, 3.7) {\contour{white}{$P'_3$}};
	\end{pgfonlayer}
	\begin{pgfonlayer}{edgelayer}
	\contourlength{2pt}
		\fill[gray!20] (18.center)--(19.center)--(20.center)--(21.center)--(22.center)--(23.center)--(12.center)--cycle;
	\fill[gray!30] (10.center)--(13.center)--(15.center)--(14.center)--(16.center)--(3.center)--cycle;
	\fill[gray!50] (8.center)--(7.center)--(17.center)--(2.center)--cycle;
		\draw[b] (8,0) to (-8,0);
		\draw [b] (8) to (7);
		\draw [style=b] (7) to (17);
		\draw [style=b] (17) to (2);
		\draw [style=b] (8) to (7);
		\draw [style=b] (10) to (13);
		\draw [style=b] (13) to (15);
		\draw [style=b] (15) to (14);
		\draw [style=b] (14) to (16);
		\draw [style=b] (16) to (3);
		\draw [style=b] (4) to (3);
		\draw [style=b] (18) to (19);
		\draw [style=b] (19) to (20);
		\draw [style=b] (20) to (21);
		\draw [style=b] (18) to (12);
		\draw [style=b] (21) to (22);
		\draw [style=b] (22) to (23);
		\draw [style=b] (24) to (12);
		
		\draw [style=cont] (8) to (7);
		\draw [style=cont] (7) to (17);
		\draw [style=cont] (17) to (2);
		\draw [style=cont] (8) to (7);
		\draw [style=cont] (10) to (13);
		\draw [style=cont] (13) to (15);
		\draw [style=cont] (15) to (14);
		\draw [style=cont] (14) to (16);
		\draw [style=cont] (16) to (3);
		\draw [style=cont] (4) to (3);
		\draw [style=cont] (18) to (19);
		\draw [style=cont] (19) to (20);
		\draw [style=cont] (20) to (21);
		\draw [style=cont] (18) to (12);
		\draw [style=cont] (21) to (22);
		\draw [style=cont] (22) to (23);
		\draw [style=cont] (24) to (12);
		\draw [line width=9pt, red] (33) -- (32) node [midway] {\contour{white}{$3$}};
		\draw [line width=5pt, red] (30) -- (27) node [midway] {\contour{white}{$1$}};
		\draw [line width=9pt, red] (27) -- (26) node [midway] {\contour{white}{$2$}};
		\draw [line width=5pt, blue] (26) -- (25) node [midway] {\contour{white}{$1$}};
		\draw [line width=9pt, blue] (29) -- (28) node [midway] {\contour{white}{$2$}};
		\draw [line width=5pt, blue] (28) -- (31) node [midway] {\contour{white}{$2$}};
		\draw[dashed] (2)--(2,-1)--(-2,-1)--(8);
		\draw[dashed] (4)--(4,-1.3)--(-4,-1.3)--(10);
		\draw[dashed] (23)--(7,-1.6)--(-7,-1.6)--(24);
		\draw[] (8,0) -- (-8,0);
		\draw[root] (0)--(1);
		\draw[dashed] (-9,0)--(-8,0) (9,0)--(8,0);
		\draw[->, very thick] (0,-1.8)--(0,-2.5);
	\end{pgfonlayer}
\end{tikzpicture}\\
\begin{tikzpicture}[cont/.style={ultra thick}, scale=.7]
	\begin{pgfonlayer}{nodelayer}
		\node [style=real] (0) at (-3, -0) {};
		\node [style=real] (1) at (-2, -0) {};
		\node [style=real] (2) at (-1, -0) {};
		\node [style=real] (3) at (0, -0) {};
		\node [style=real] (4) at (1, -0) {};
		\node [style=real] (5) at (2, -0) {};
		\node [style=real] (6) at (3, -0) {};
		\node [style=real, label=below:$x_7$] (7) at (4, -0) {};
		\node [style=real] (8) at (-4, -0) {};
		\node [style=real] (9) at (-2, 1.75) {};
		\node [style=real] (10) at (-4.75, 1.25) {};
		\node [style=real] (11) at (-4.25, -1.75) {};
		\node [style=real] (12) at (-1.25, -1.5) {};
		\node [style=real] (13) at (0.5, 2.25) {};
		\node [style=real] (14) at (-5, 3.25) {};
		\node [style=real] (15) at (-6.25, -1.25) {};
		\node [style=real] (16) at (-3, -3.5) {};
		\node [style=real] (17) at (2.5, -1.5) {};
		\contourlength{1pt}
		\node [style=none] (18) at (-3.75, 0.75) {\contour{white}{$P'_1$}};
		\node [style=none] (19) at (-3.5, 1.75) {\contour{white}{$P'_2$}};
		\node [style=none] (20) at (-2, 3) {\contour{white}{$P'_3$}};
	\end{pgfonlayer}
	\begin{pgfonlayer}{edgelayer}
	\fill[gray!20] (7.center)--(13.center)--(14.center)--(15.center)--(16.center)--(17.center)--(6.center)--cycle;
	\fill[gray!30] (3.center)--(9.center)--(10.center)--(11.center)--(12.center)--(4.center)--(3.center);
	\fill[gray!50] [bend right=45, looseness=1.25] (2.center) to (8.center) [bend right=60, looseness=1.00] to (1.center)--(2.center);
	\draw[b] (0,0)--(6,0);
	\draw (0) to (5,0);
	\draw[dashed] (5,0)--(6,0);
	
		\draw [style=b, bend right=45, looseness=1.25] (2) to (8);
		\draw [style=b, bend right=60, looseness=1.00] (8) to (1);
		\draw [style=b] (1) to (2);
		\draw [style=b] (3) to (9);
		\draw [style=b] (9) to (10);
		\draw [style=b] (10) to (11);
		\draw [style=b] (11) to (12);
		\draw [style=b] (12) to (4);
		\draw [style=b] (3) to (4);
		\draw [style=b] (7) to (13);
		\draw [style=b] (13) to (14);
		\draw [style=b] (14) to (15);
		\draw [style=b] (15) to (16);
				\draw [style=b] (16) to (17);
		\draw [style=b] (17) to (6);
		\draw [style=b] (6) to (7);
		
	\draw[root] (0)--(1);
		\draw [style=cont, bend right=45, looseness=1.25] (2) to (8);
		\draw [style=cont, bend right=60, looseness=1.00] (8) to (1);
		\draw [style=cont] (1) to (2);
		\draw [style=cont] (3) to (9);
		\draw [style=cont] (9) to (10);
		\draw [style=cont] (10) to (11);
		\draw [style=cont] (11) to (12);
		\draw [style=cont] (12) to (4);
		\draw [style=cont] (3) to (4);
		\draw [style=cont] (7) to (13);
		\draw [style=cont] (13) to (14);
		\draw [style=cont] (14) to (15);
		\draw [style=cont] (15) to (16);
				\draw [style=cont] (16) to (17);
		\draw [style=cont] (17) to (6);
		\draw [style=cont] (6) to (7);	
	\end{pgfonlayer}
\end{tikzpicture}
	\caption{\label{fig:fences one-ended}The fences $(P_i)_{i\geq 1}$ are the subpaths of the fences $(P_i')_{i\geq 1}$ above obtained by disregarding boundary edges; we have $X_1^+=2$, $X_2^+=1$, $X_3^+=3$, $X_1^-=1$, $X_2^-=2$, $X_3^-=2$. The overshoots of $P_i'$ are both equal to $\max\{X_i^+,X_i^-\}$ (hence 2, 2, 3 for $i=1,2,3$), and $r(i)$ (see the folded version) is $X_1+\ldots+X_i$ (e.g.~we have $r(3)=7$). }
\end{figure}

Consider first $\Q_\infty^\rightarrow$, built from $\UIHPQ$ by glueing the boundary onto itself as described in the Introduction. Start by setting $k=1$, reroot $\UIHPQ$ at $e_0=(x_{-1},x_0)$, and build a fence $P_1$ avoiding $x_0$ whose endpoints are $x_{-\ell_1}$ and $x_{r_1}$ as described in Section~\ref{algorithm}; let $X^+_1=r_1$ and $X^-_1=\ell_1$ be its right and left overshoots respectively. Before building $P_{2}$ we set $P'_1$ to be $P_1$ with the addition of the portion of the boundary between $x_{-r_1}$ and $x_{-\ell_1}$ if $r_1>\ell_1$, or between $x_{r_1}$ and $x_{\ell_1}$ if $r_1<\ell_1$. Hence the path $P'_{1}$ connects symmetric vertices on the boundary. We then consider the map $\UIHPQ(1)$ obtained by erasing the region of $\UIHPQ$ lying below $P'_{1}$ (or, equivalently, below $P_{1}$), rooted at the first boundary edge on the left of $P'_{1}$. The map $\UIHPQ(1)$ is a copy of the UIHPQ independent of the part erased. We then build $P_2$ by running the algorithm from Section \ref{algorithm} inside $ \UIHPQ(1)$, setting $k$ to be the number of edges of $P'_{1}$ plus $1$. We then extend $P_{2}$ into $P_{2}'$ as above to make it connect mirror vertices. Iterating the process we build disjoint paths $P'_{i}$ connecting mirror vertices $x_{-r(i)}$ to $x_{r(i)}$.

Consider now the sequence $(P_i')_{i\geq 1}$ as seen in $\Q_\infty^\rightarrow$ (Figure~\ref{fig:fences one-ended}): the fences now form nested disjoint loops. By planarity if $j \geq r(i)$ then the $j$-th point on the distinguished self-avoiding walk $ \mathsf{P}^{\rightarrow}$ is at distance at least $i$ from the origin in $\Q_{\infty}^{\rightarrow}$. Our lower bound \eqref{prop:lower bound} (LHS) is then implied if we can show that   \begin{eqnarray} \label{eq:goal1} r(n) \preceq n^{2}.  \end{eqnarray}
If we denote by $X_{i}^{-}$ and $X^{+}_{i}$ the left and right overshoots of $P_{i}$, then 
$$r(n) = \sum_{j=1}^{n} \max( X_{j}^{+}, X_{j}^{-}).$$
Notice that the overshoots $X^+_1, X^-_1, \ldots,X^+_i, X^-_i$ are \emph{not} independent nor identically distributed; however, using Corollary \ref{cor:coupling} we can couple those overshoots with a sequence of i.i.d.~random variables $(O_{i})_{i \geq 1}$ having the law prescribed in Corollary \ref{cor:coupling} so that  $\max( X_{j}^{+}, X_{j}^{-}) \leq O_{j}$ for all $j \geq 1$. Hence in this coupling we have $r(n) \leq O_{1} + \cdots + O_{n}$. Standard estimates for i.i.d.~variables with heavy tails then show that $O_{1} + \cdots + O_{n} \preceq n^{2}$ which proves our goal \eqref{eq:goal1} which -- combined with Corollary~\ref{cor:old bound} -- proves the first part of Theorem~\ref{thm:diffusive}.

\bigskip

The case of $\Q_\infty^\leftrightarrow$ is essentially the same. Supposing $\UIHPQ$ and $\UIHPQ'$ are the two independent ``halves'' sharing the distinguished path $(\mathsf{P}^{\leftrightarrow}_i)_{i\in \mathbb{Z}}$ on the vertices $(x_i)_{i\in\mathbb{Z}}$, one proceeds to build a fence $P_1$ within $\UIHPQ$ that avoids $x_0$, then a fence $Q_1$ in $\UIHPQ'$ that does the same; one then builds $P_1'$ and $Q_1'$ by adding boundary edges to $P_1$ and $Q_1$ so that the endpoints of $P_1'$ coincide with those of $Q_1'$ (see Figure~\ref{fig:fences1}), and thus the right overshoot of $P_1'$ is the maximum between the right overshoot of $P_1$ and the left overshoot of $Q_1$. Iterating such a construction and applying the very same estimates as for the case of $\Q_\infty^\rightarrow$ finally shows Theorem~\ref{thm:diffusive}.

     \section{Volume estimates and singularity}
     \subsection{Proof of Corollary \ref{cor:volumegrowth}}
\proof[Proof of Corollary \ref{cor:volumegrowth}]  We treat the case of $\Q_{\infty}^\rightarrow$, the argument being similar in the case of $\Q_{\infty}^\leftrightarrow$. We assume that $\Q_{\infty}^\rightarrow$ has been constructed from $\UIHPQ$ by folding its boundary onto itself. Since the surgery operation performed to create $\Q_{\infty}^\rightarrow$ from $\UIHPQ$ can only decrease distances we have $\Ball_{r}(\UIHPQ) \subseteq  \Ball_{r}(\Q_{\infty}^\rightarrow)$ and it follows from the estimates on the volume growth in the UIHPQ \cite[Proposition 6.2]{CCuihpq} that 
      $$ \#  \Ball_{r}(\Q_{\infty}^\rightarrow) \succeq r^4.$$
      Let us now turn to the upper bound.      Consider the portion of the folded boundary of $\UIHPQ$ that is inside $ \Ball_{r}(\Q_{\infty}^\rightarrow)$. By Theorem \ref{thm:diffusive} the length $L_{r}$ of this portion is $\preceq r^2$. Hence it is immediate that $ \Ball_{r}(\Q_{\infty}^\rightarrow)$ is contained in the set of all faces having one vertex at distance at most $r$ from the boundary $\llbracket -L_{r}, L_{r} \rrbracket$ in $\UIHPQ$. Let us denote by $M_{r}$ the maximal distance in $\UIHPQ$ to the origin of this piece of boundary. By the above argument we have 
      $$  \Ball_{r}(\Q_{\infty}^\rightarrow) \subseteq  \Ball_{M_{r}+r}(\UIHPQ).$$
But applying the distance estimates along the boundary inside the UIHPQ (\cite[Proposition 6.1]{CCuihpq}) one deduces that $M_{r} \approx  \sqrt{L_{r}} \approx r$ and using volume estimates once more we get $\# \Ball_{M_{r}+r}(\UIHPQ) \approx r^4$, which completes the proof of the upper bound and hence of Corollary~\ref{cor:volumegrowth}. \endproof

  \begin{figure}
\centering
\begin{tikzpicture}[every path/.style={}]
\node (P1) at (-0.3,0.9) {\contour{white}{\color{red}$P_1'$}};
\node (P2) at (-0.7,1.5) {\contour{white}{\color{red}$P_2'$}};
\node (P1) at (-0.3,2.5) {\contour{white}{\color{red}$P_3'$}};
\node (P1) at (-0.1,-1) {\contour{white}{\color{blue}$Q_1'$}};
\node (P2) at (-1.4,-1.5) {\contour{white}{\color{blue}$Q_2'$}};
\node (P1) at (-0.3,-2.5) {\contour{white}{\color{blue}$Q_3'$}};
	\begin{pgfonlayer}{nodelayer}
		\node [style=real] (0) at (0, -0) {};
		\node [style=real] (1) at (1, -0) {};
		\node [style=real] (2) at (2, -0) {};
		\node [style=real] (3) at (3, -0) {};
		\node [style=real] (4) at (4, -0) {};
		\node [style=real, label=below:$x_5$] (5) at (5, -0) {};
		\node [style=real] (6) at (-1, -0) {};
		\node [style=real] (7) at (-2, -0) {};
		\node [style=real] (8) at (-3, -0) {};
		\node [style=real] (9) at (-4, -0) {};
		\node [style=real] (10) at (-5, -0) {};
		\node [style=real] (11) at (-6, -0) {};
		\node [style=real] (12) at (-2.5, 1.5) {};
		\node [style=real] (13) at (1.25, 1) {};
		\node [style=real] (14) at (2, 2.5) {};
		\node [style=real] (15) at (0.5, 2.75) {};
		\node [style=real] (16) at (3.5, 1.25) {};
		\node [style=real] (17) at (3, 1.75) {};
		\node [style=real] (18) at (2.75, 3) {};
		\node [style=real] (19) at (4.75, 0.5) {};
		\node [style=real] (20) at (3.5, 0.5) {};
		\node [style=real] (21) at (2.25, 0.75) {};
		\node [style=real] (22) at (0, 1.5) {};
		\node [style=real] (23) at (1.25, 1.5) {};
		\node [style=real] (24) at (0, 0.5) {};
		\node [style=real] (25) at (-0.25, 2) {};
		\node [style=real] (26) at (-3.5, 1) {};
		\node [style=real] (27) at (-4.75, 1.25) {};
		\node [style=real] (28) at (-5, 0.25) {};
		\node [style=real] (29) at (-1.25, 2.5) {};
		\node [style=real] (30) at (-0.75, 0.5) {};
		\node [style=real] (31) at (-4.25, -1) {};
		\node [style=real] (32) at (-2.5, -1.25) {};
		\node [style=real] (33) at (-0.7, -1.79) {};
		\node [style=real] (34) at (-0.25, -0.5) {};
		\node [style=real] (35) at (0.5, -0.5) {};
		\node [style=real] (36) at (-0.25, -2) {};
		\node [style=real] (37) at (-1.75, -0.75) {};
		\node [style=real] (38) at (-0.75, -1.5) {};
		\node [style=real] (39) at (2, -1.25) {};
		\node [style=real] (40) at (3, -0.75) {};
		\node [style=real] (41) at (4, -1.25) {};
		\node [style=real] (42) at (3.75, -1.75) {};
		\node [style=real] (43) at (-2.25, -2.75) {};
		\node [style=real] (44) at (1.5, -3) {};
		\node [style=real] (45) at (1.75, -2.25) {};
		\node [style=real] (46) at (0.5, -1.5) {};
		\node [style=real] (47) at (-0.75, -0.75) {};
		\node [style=real] (48) at (-3.5, -0.25) {};
		\node [style=real] (49) at (-3.25, -2) {};
		\node [style=real] (50) at (-5.25, -1.5) {};
	\end{pgfonlayer}
	\begin{pgfonlayer}{edgelayer}
		\draw[ultra thick] (11) to (10);
		\draw [style=fence] (10) to (9);
		\draw [bend right=75, looseness=1.25] (9) to (8);
		\draw[ultra thick] (7) to (6);
		\draw[ultra thick] (6) to (0);
		\draw[ultra thick] (1) to (2);
		\draw[ultra thick] (2) to (3);
		\draw [style=fence] (3) to (4);
		\draw[ultra thick] (4) to (5);
		\draw [bend left, looseness=1.25] (30) to (1);
		\draw (24) to (1);
		\draw (13) to (1);
		\draw (30) to (23);
		\draw [bend left=45, looseness=1.25] (25) to (21);
		\draw (25) to (8);
		\draw [bend left=45, looseness=1.25] (22) to (21);
		\draw (8) to (30);
		\draw [style=fence] (9) to (26);
		\draw [style=fence] (26) to (12);
		\draw (12) to (8);
		\draw [style=fence] (12) to (29);
		\draw [style=fence] (29) to (25);
		\draw (21) to (16);
		\draw (16) to (19);
		\draw [bend right=60, looseness=1.50] (3) to (16);
		\draw (4) to (19);
		\draw (3) to (20);
		\draw [style=fence] (16) to (3);
		\draw [bend left=60, looseness=1.25] (16) to (5);
		\draw (14) to (21);
		\draw [style=fence] (14) to (17);
		\draw [style=fence] (17) to (16);
		\draw [style=fence] (25) to (15);
		\draw [style=fence] (15) to (14);
		\draw [style=fence] (6) to (30);
		\draw [style=fence] (30) to (1);
		\draw [style=fence] (8) to (22);
		\draw [style=fence] (22) to (23);
		\draw [style=fence] (23) to (13);
		\draw [style=fence] (13) to (21);
		\draw [style=fence] (21) to (2);
				\tikzset{fence/.append style={blue}};
						\draw [style=fence] (7) to (8);
		\draw [style=fence] (1) to (35);
		\draw [style=fence] (35) to (34);
		\draw [style=fence] (34) to (47);
		\draw [style=fence] (47) to (6);
		\draw (0) to (34);
		\draw [bend right=45, looseness=1.00] (34) to (2);
		\draw [style=fence, bend left=15, looseness=1.00] (2) to (46);
		\draw [style=fence] (46) to (38);
		\draw (38) to (34);
		\draw [style=fence] (38) to (37);
		\draw (37) to (47);
		\draw [style=fence] (37) to (7);
		\draw [style=fence] (4) to (40);
		\draw (2) to (40);
		\draw (40) to (39);
		\draw (46) to (39);
		\draw (46) to (45);
		\draw [style=fence] (45) to (40);
		\draw [bend right=75, looseness=1.25] (38) to (36);
		\draw [style=fence] (36) to (45);
		\draw (33) to (36);
		\draw (38) to (36);
		\draw (8) to (32);
		\draw (32) to (37);
		\draw [style=fence, bend right=45, looseness=1.00] (32) to (36);
		\draw [style=fence] (32) to (10);
		\draw (48) to (8);
		\draw[ultra thick] (9) to (8);
		\draw (41) to (5);
		\draw (41) to (40);
		\draw (45) to (42);
		\draw (42) to (41);
		\draw (10) to (31);
		\draw (11) to (50);
		\draw (50) to (31);
		\draw (11) to (28);
		\draw (28) to (9);
		\draw (28) to (27);
		\draw (27) to (26);
		\draw (31) to (49);
		\draw (49) to (32);
		\draw (27) to (29);
		\draw (15) to (18);
		\draw (18) to (17);
		\draw (49) to (43);
		\draw [bend right=15, looseness=1.00] (43) to (36);
		\draw (43) to (44);
		\draw (44) to (45);
		\draw[ultra thick, ->] (0) to (1);
	\end{pgfonlayer}
\end{tikzpicture}
\caption{\label{fig:fences1} Illustration of the construction of the fences in the case of $\Q_{\infty}^{\leftrightarrow}$.}
\end{figure}
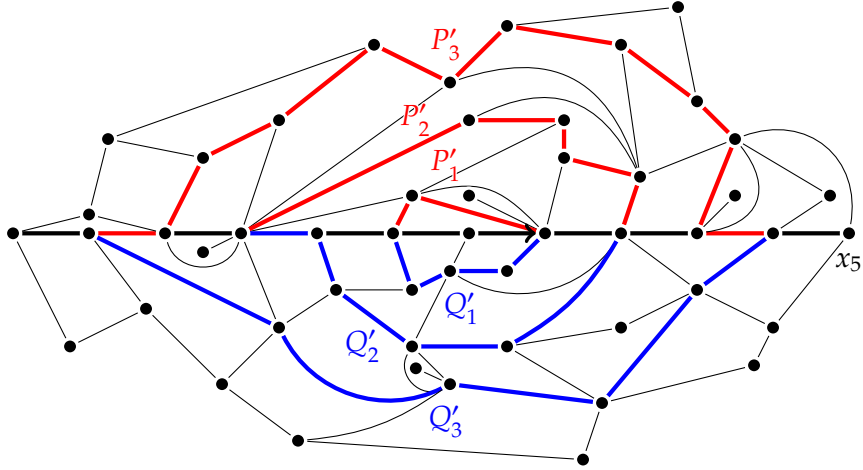

      \subsection{Proof of Theorem \ref{thm:singular}}

We first recall the scaling limit results that we will need. The work of Le Gall \& M\'enard on the UIPQ \cite{LGM10,LGM10erratum}  as well as our previous work on the UIHPQ \cite{CCuihpq} show that there are two random variables $ \mathcal{V}_{p}$ and $ \mathcal{V}_{h}$ such that 
 \begin{eqnarray*} r^{-4} \# \Ball_{r}(\Q_{\infty})  \to \mathcal{V}_{p} \quad \mbox{ and } \quad r^{-4} \# \Ball_{r}(\UIHPQ)  \to \mathcal{V}_{h},  \end{eqnarray*}
  in distribution as $r \to \infty$. The expectations of such continuous random variables have been computed and in particular $\mathbb{E}[\mathcal{V}_{h}]=\frac{7}{9} \mathbb{E}[\mathcal{V}_{p}]$. We will only use a trivial consequence of these calculations: if $ \mathcal{V}_{h}'$ is an (independent) copy of $ \mathcal{V}_{h}$ since $ \mathbb{E}[  \mathcal{V}_{h} +  \mathcal{V}'_{h}] >  \mathbb{E} [\mathcal{V}_{p}]$ one can find $\alpha>0$ such that 
  \begin{eqnarray}  \label{eq:choixx0}\mathbb{P}(\mathcal{V}_{h} +  \mathcal{V}_{h}'> \alpha) > \mathbb{P}(\mathcal{V}_{p} > \alpha),  \end{eqnarray}
we can and will furthermore assume that $\alpha$ is not an atom for the law of $ \mathcal{V}_{ p}$ nor for that of $ \mathcal{V}_{h} + \mathcal{V}_{h}'$ (this is possible since there are at most a countable number of atoms for each law).
  
  By construction, the volume of the ball of radius $r$ inside   $\Q_{\infty}^{\leftrightarrow}$ is at least $\# \widetilde{\Ball}_{r}(\UIHPQ)+\#\widetilde{ \Ball}_{r}(\UIHPQ')$, where $\UIHPQ$ and $\UIHPQ'$ are the two independent copies of the UIHPQ with a simple boundary used to construct $\Q_{\infty}^{\leftrightarrow}$, and $ \widetilde{\Ball}_{r}(\UIHPQ)$ is the set of inner vertices (not on the boundary) which are at distance less than $r$ from the origin of the map (since we shall use estimates for $r^{-4} \# \Ball_{r}(\UIHPQ)$, the number of vertices on the boundary which is of order $r^{2}$ will turn out to be completely irrelevant, see \cite[Section 6]{CCuihpq}). In particular from the scaling limit results recalled in the beginning of this section we deduce that for the $\alpha$ chosen in \eqref{eq:choixx0}  we have
\begin{eqnarray} \label{eq:laws converge} 
\lim_{r \to \infty }\mathbb{P}(r^{-4}\# \Ball_{r}(\Q_{\infty}) > \alpha)  &=&  \mathbb{P}( \mathcal{V}_{p} >\alpha ), \\
 \liminf_{r \to \infty }\mathbb{P}(r^{-4}\# \Ball_{r}(\Q_{\infty}^{\leftrightarrow}) > \alpha)  &\geq& \liminf_{r \to \infty }\mathbb{P}(r^{-4}(\# \widetilde{\Ball}_{r}(\UIHPQ)+\#\widetilde{ \Ball}_{r}(\UIHPQ ')) > \alpha)\nonumber \\ &=&  \mathbb{P}( \mathcal{V}_{h}+\mathcal{V}_{h}' >\alpha ).\label{eq:laws converge2}   \end{eqnarray}

Given a quadrangulation of the plane $ \mathfrak{q}$, set $ \mathcal{X}_{r}( \mathfrak{q})=1$ if  $\# \Ball_{r}( \mathfrak{q}) > \alpha r^4$ and $\mathcal{X}_{r}(q)=0$ otherwise. Similarly if $ \mathfrak{h}_{1}$ and $\mathfrak{h}_{2}$ are two quadrangulations of the half-plane then we set $ \mathcal{Y}_{r}( \mathfrak{h}_{1}, \mathfrak{h}_{2}) =1$ if $$\# \widetilde{\Ball}_{r}( \mathfrak{h}_{1}) + \# \widetilde{\Ball}_{r}( \mathfrak{h}_{2}) > \alpha r^4.$$ Clearly with this notation we have $ \mathcal{X}_{r}( \Q_{\infty}^{\leftrightarrow}) \geq \mathcal{Y}_{r}( \UIHPQ, \UIHPQ')$.  The singularity result follows from an evaluation of the random variables $ \mathcal{X}_{r}(\Q_{\infty}^{\leftrightarrow})$ and $\mathcal{X}_{r}(\Q_{\infty})$ at different scales $0 \ll r_{1} \ll r_{2} \ll r_{3} \ll \cdots$ chosen in such a way that $ \mathcal{X}_{r_{i}}(\Q_{\infty})$ is roughly independent of $ \mathcal{X}_{r_{j}}(\Q_{\infty})$ for $i\neq j$, so that one may invoke a law of large numbers. To make this precise we first state an independence lemma which we prove at the end of the paper.
\begin{lemma}[Independance of scales]  \label{lem:scales}For any $r \geq 0$ and any $ \varepsilon >0$ there exists $R \geq r$ such that for any $s \geq R$ we have 
$$ \mathrm{Cov}\big( \mathcal{Y}_{r}(\UIHPQ,\UIHPQ') ;  \mathcal{Y}_{s}( \UIHPQ, \UIHPQ')\big) \leq \varepsilon,$$
$$ \mathrm{Cov}\big( \mathcal{X}_{r}(\Q_{\infty}) ;  \mathcal{X}_{s}( \Q_{\infty})\big) \leq \varepsilon.$$

\end{lemma}
Using the above lemma we build a sequence $0 \ll r_{1} \ll r_{2} \ll \cdots$ such that for all $i <j$ one has $ \mathrm{Cov}( \mathcal{X}_{r_{i}}(	\Q_{\infty}),  \mathcal{X}_{r_{j}}(\Q_{\infty})) \leq 2^{-j}$ as well as $\mathrm{Cov}\big( \mathcal{Y}_{r_{i}}(\UIHPQ,\UIHPQ') ;  \mathcal{Y}_{r_{j}}( \UIHPQ, \UIHPQ')\big) \leq 2^{{-j}}$. Hence the hypotheses for the strong law of large numbers for weakly correlated variables are satisfied, see e.g.~\cite{Lyons88}, and this implies that 
$$ \frac{1}{k} \sum_{i=1}^k  \mathcal{X}_{r_{i}}(\Q_{\infty}) \xrightarrow[k\to\infty]{a.s.} \mathbb{P}( \mathcal{V}_{p}> \alpha) \quad \mbox{and }\quad \frac{1}{k} \sum_{i=1}^k\mathcal{Y}_{r_{i}}(\UIHPQ,\UIHPQ') \xrightarrow[k\to\infty]{a.s.} \mathbb{P}( \mathcal{V}_{h}+\mathcal{V}'_{h}> \alpha).$$
Since we have $ \mathcal{X}_{r}( \Q_{\infty}^{\leftrightarrow}) \geq \mathcal{Y}_{r}( \UIHPQ, \UIHPQ')$ this entails thanks to \eqref{eq:choixx0}
$$\liminf_{k\to\infty} \frac{1}{k} \sum_{i=1}^k  \mathcal{X}_{r_{i}}(\Q_\infty^{\leftrightarrow}) \geq \mathbb{P}( \mathcal{V}_{h}+\mathcal{V}'_{h}> \alpha)>\mathbb{P}( \mathcal{V}_{p}> \alpha).$$
In other words, the event $\{ \liminf_{k \to \infty} k^{-1} \sum_{i=1}^{k} \mathcal{X}_{r_{i}}( \mathfrak{q}) = \mathbb{P}( \mathcal{V}_{p} >\alpha)\}$ has probability $1$ under the law of the UIPQ and probability $0$ under the law of $\Q_{\infty}^{\leftrightarrow}$, finally establishing singularity of the two distributions, as expected. 

\subsection{Lemma \ref{lem:scales}: decoupling the scales}
\begin{proof}[Proof of Lemma \ref{lem:scales}]

We begin with the first statement concerning the half-planes. Let $ \UIHPQ$ and $\UIHPQ'$ be two copies of the UIHPQ and let $r \geq 0$. As recalled above, in \cite[Section 6]{CCuihpq} we established scaling limits for the volume process $(\# \Ball_{r}(\UIHPQ))_{r \geq 0}$ in the UIHPQ. It follows in particular from the almost sure continuity of the scaling limit process at time $t=1$ and the fact that the boundary effects are negligible that for any function $o(r)$ negligible with respect to $r$ we have 
 \begin{eqnarray} \label{eq:petit} \frac{\# \Ball_{r+o(r)}(\UIHPQ) }{\# \widetilde{\Ball}_{r}(\UIHPQ)} \xrightarrow[r\to\infty]{( \mathbb{P})} 1,  \end{eqnarray} where we recall that $\widetilde{\Ball}_{r}(\UIHPQ)$ are the inner vertices of ${\Ball}_{r}(\UIHPQ)$. Recall also that we denoted by $ \Ball_{r}^{\bullet}( \UIHPQ)$ the hull of the ball of radius $r$ inside $ \UIHPQ$, and that by the spatial Markov property of the UIHPQ the map $\UIHPQ[r]:=\UIHPQ \backslash \Ball_{r}^{\bullet}( \UIHPQ)$ rooted, say, at the first edge on the boundary of $\UIHPQ$ on the left of $ \Ball_{r}^{\bullet}(\UIHPQ)$, is distributed as a UIHPQ and independent of $\Ball_{r}^{\bullet}(\UIHPQ)$ (and also of $\# \widetilde{ \Ball}_{r}( \UIHPQ)$). An easy geometric argument shows that there exist two random constants $A_{r},B_{r} \geq 0$ depending on $\Ball_{r}^{\bullet}(\UIHPQ)$ such that we have for all $s \geq r$
$$ \#\widetilde{\Ball}_{s-A_{r}}(\UIHPQ[r]) - B_{r} \leq \#\widetilde{\Ball}_{s}( \UIHPQ) \leq  \#\widetilde{\Ball}_{s+A_{r}}(\UIHPQ[r]) + B_{r}.$$ Hence, using \eqref{eq:petit} we deduce that  for any $r \geq 1$
$$ \frac{\# \widetilde{\Ball}_{s}( \UIHPQ)}{\# \widetilde{\Ball}_{s}(\UIHPQ[r])}  \xrightarrow[s\to\infty]{( \mathbb{P})} 1,$$ and similarly when considering the other copy $ \UIHPQ'$ of the UIHPQ. The point being that now $\# \widetilde{\Ball}_{s}(\UIHPQ[r])$ is independent of $\# \widetilde{\Ball}_{r}(\UIHPQ)$. We use these convergences together with the fact that $\alpha$ is not an atom of the law of $ \mathcal{V}_{h}+ \mathcal{V}_{h}'$ to deduce that
$$ \mathbf{1}_{\# \widetilde{\Ball}_{s}( \UIHPQ) + \# \widetilde{\Ball}_{s}( \UIHPQ') > \alpha s^{4}} - \mathbf{1}_{\# \widetilde{\Ball}_{s}( \UIHPQ[r]) + \# \widetilde{\Ball}_{s}( \UIHPQ'[r]) > \alpha s^{4}} \xrightarrow[s\to\infty]{( \mathbb{P})} 0.$$ 
When developing the covariance $\mathrm{Cov}( \mathcal{Y}_{r}(\UIHPQ,\UIHPQ'); \mathcal{Y}_{s}(\UIHPQ,\UIHPQ'))$ we can then replace $\mathbf{1}_{\# \widetilde{\Ball}_{s}( \UIHPQ) + \# \widetilde{\Ball}_{s}( \UIHPQ') > \alpha s^{4}}$ by $\mathbf{1}_{\# \widetilde{\Ball}_{s}( \UIHPQ[r]) + \# \widetilde{\Ball}_{s}( \UIHPQ'[r]) > \alpha s^{4}}$ with asymptotically no harm: For $r$ fixed as $s \to \infty$ we have
\begin{eqnarray*} && \mathrm{Cov}\big( \mathcal{Y}_{r}(\UIHPQ,\UIHPQ') ;  \mathcal{Y}_{s}( \UIHPQ, \UIHPQ')\big) \\ 
 &=& \mathbb{E}\left[   \mathbf{1}_{\# \widetilde{\Ball}_{s}( \UIHPQ) + \# \widetilde{\Ball}_{s}( \UIHPQ') > \alpha s^{4}}\mathbf{1}_{\# \widetilde{\Ball}_{r}( \UIHPQ) + \# \widetilde{\Ball}_{r}( \UIHPQ') > \alpha r^{4}}\right] \\ 
& & - \mathbb{P}(\# \widetilde{\Ball}_{s}( \UIHPQ) + \# \widetilde{\Ball}_{s}( \UIHPQ') > \alpha s^{4}) \mathbb{P}(\# \widetilde{\Ball}_{r}( \UIHPQ) + \# \widetilde{\Ball}_{r}( \UIHPQ') > \alpha r^{4})\\ 
&=&
o(1)+ \mathbb{E}\left[   \mathbf{1}_{\# \widetilde{\Ball}_{s}( \UIHPQ[r]) + \# \widetilde{\Ball}_{s}( \UIHPQ'[r]) > \alpha s^{4}}\mathbf{1}_{\# \widetilde{\Ball}_{r}( \UIHPQ) + \# \widetilde{\Ball}_{r}( \UIHPQ') > \alpha r^{4}}\right] \\ && - \mathbb{P}(\# \widetilde{\Ball}_{s}( \UIHPQ) + \# \widetilde{\Ball}_{s}( \UIHPQ') > \alpha s^{4}) \mathbb{P}(\# \widetilde{\Ball}_{r}( \UIHPQ) + \# \widetilde{\Ball}_{r}( \UIHPQ') > \alpha r^{4})\\
&=& o(1) + \mathbb{P}(\# \widetilde{\Ball}_{s}( \UIHPQ[r]) + \# \widetilde{\Ball}_{s}( \UIHPQ'[r]) > \alpha s^{4}) \mathbb{P}(\# \widetilde{\Ball}_{r}( \UIHPQ) + \# \widetilde{\Ball}_{r}( \UIHPQ') > \alpha r^{4})  \\
&& -\mathbb{P}(\# \widetilde{\Ball}_{s}( \UIHPQ) + \# \widetilde{\Ball}_{s}( \UIHPQ') > \alpha s^{4}) \mathbb{P}(\# \widetilde{\Ball}_{r}( \UIHPQ) + \# \widetilde{\Ball}_{r}( \UIHPQ') > \alpha r^{4}),\end{eqnarray*}
where in the last line we used the independence of $\UIHPQ[r]$ and $ \Ball^{\bullet}_{r}(\UIHPQ)$ (and similarly for $\UIHPQ'$). Since $ \UIHPQ[r]$ has the law of a UIHPQ the product of probabilities cancels out and the covariance indeed tends to $0$ as $s \to \infty$ as desired.

For the case of the UIPQ things are a little more complicated since the spatial Markov property involves the perimeter of the discovered region. For $r \geq 1$ we also consider the hull $\Ball_{r}^{\bullet}(\Q_{\infty})$ of the ball of radius $r$ inside the UIPQ. Unfortunately $\Q_{\infty} \backslash \Ball_{r}^{\bullet}(\Q_{\infty})$ (appropriately rooted) is not independent of $\Ball_{r}^{\bullet}(\Q_{\infty})$: conditionally on $\Ball_{r}^{\bullet}(\Q_{\infty})$ the map $\Q_{\infty} \backslash \Ball_{r}^{\bullet}(\Q_{\infty})$ has the law of a UIPQ of the $2\ell$-gon where $2\ell$ is the perimeter of the unique hole of $\Ball_{r}^{\bullet}(\Q_{\infty})$. Yet, it can be shown (for example using the techniques of \cite{CLGpeeling}) that for any $\ell \geq 1$ if $\Q_{\infty,\ell}$ is a UIPQ of the $2\ell$-gon then we still have 
$$ r^{-4} \# \Ball_{r}(\Q_{\infty,\ell}) \xrightarrow[r\to\infty]{(d)} \mathcal{V}_{p}.$$
This is in fact sufficient in order to adapt the above proof to the case of the UIPQ. At this point in the paper, we leave the details to the courageous reader.
\end{proof}

\section{Open problems} \label{sec:conj}
{In this section we discuss several open problems related to the topic of this paper, and indicate some possible directions for further research. First we state a natural conjecture motivated by Theorem \ref{thm:singular} (see \cite[Section 5]{ANR14} for a related conjecture):

\begin{conj} \label{conj:singular} The laws of  $\Q_{\infty}$, $ \Q^\rightarrow_\infty$ and of $\Q_{\infty}^\leftrightarrow$  are singular with respect to each other.
\end{conj}

\setcounter{open}{0}
As stated in the Introduction, a possible way towards a proof of this conjecture would be to use the recent work \cite{GM16a,GM16b} as a replacement of the input \eqref{eq:volumeexact} and to adapt the proof of the last section. Recall also the open question about coincidence of quenched and annealed connective constants:
\begin{open}[Coincidence of the quenched and annealed connective constants] The quenched connective constant $\mu(\Q_{\infty})$ of the UIPQ is less than the ``annealed'' connective constant which is equal to $9/2$. Do we actually have equality?
\end{open}

In light of the works devoted to random walks on random planar maps (and in particular the fact that the UIPQ is recurrent \cite{GGN12}) the following question is also natural:

\begin{open} \label{conj:recurrent} Are the random lattices $ \Q^\rightarrow_\infty$ and $\Q_{\infty}^\leftrightarrow$ almost surely recurrent?
\end{open}

Both open questions would have a positive answer if the geometry of $\Q^{\rightarrow}_{\infty}$ (resp.\ that of $\Q^\leftrightarrow_\infty$) at a given scale were comparable (in a strong ``local'' sense) to that of the UIPQ. A starting point for this strong comparison would be to show that for any $r \geq 1$ the following two random variables 
$$ \Ball_{2r}( Q_{\infty}) \backslash \Ball^{\bullet}_{r}( Q_{\infty}) \quad \mbox{ and }\quad \Ball_{2r}( Q_{\infty}^{\rightarrow}) \backslash \Ball^{\bullet}_{r}( Q_{\infty}^{\rightarrow}),$$ are contiguous (i.e.\, every graph property that holds with high probability for the first random variable also holds for the second one and vice-versa). We do not, however, intend to conjecture this is true.}

\bibliographystyle{siam}
\bibliography{bibli}

\end{document}